\documentclass[12pt,reqno]{amsproc}%
\usepackage{mathrsfs, amsfonts, amsthm, amsmath, amssymb}
\usepackage{amsmath}
\usepackage{amsfonts}
\usepackage{amssymb}
\usepackage{graphicx}%
\providecommand{\U}[1]{\protect\rule{.1in}{.1in}}
\theoremstyle{plain}
\newtheorem{theorem}{Theorem}[subsection]

\newtheorem{corollary}[theorem]{Corollary}

\newtheorem{definition}[theorem]{Definition}

\newtheorem{lemma}[theorem]{Lemma}

\newtheorem{proposition}[theorem]{Proposition}
\newtheorem{remark}[theorem]{Remark}

\makeatletter
\newcommand{\LeftEqNo}{\let\veqno\@@leqno}

\makeatother
\numberwithin{equation}  {section}
\begin{document}
\title[A stationary approach in von Neumann algebras]{A stationary approach for the Kato-Rosenblum theorem in von Neumann algebras }
\author{Qihui Li}
\curraddr{School of Mathematics, East China University of Science and Technology,
Shanghai, 200237, P. R. China}
\email{qihui\_li@126.com}
\author{Rui Wang}
\curraddr{{Shanghai Aerospace Control Technology Institute}, Shanghai, 200237, P. R. China}
\email{{18121142030@163.com}}
\thanks{The author was partly supported by NSFC (Grant No.11671133, 11871021).).}

\begin{abstract}
Let $\mathcal{M}$ be a countable decomposable properly infinite semifinite von
Neumann algebra acting on a Hilbert space $\mathcal{H}.$ An analogue of the
Kato-Rosenblum theorem in $\mathcal{M}$ has been proved in [9] by showing the
existence of generalized wave operators. It is well-known that there are two
typical approaches to show the existence of wave operators in the scattering
theory. One is called time-dependent approach and another is called stationary
approach. The main purpose of this article is to introduce a stationary
approach in $\mathcal{M}$ and then to obtain the Kato-Rosenblum theorem in
$\mathcal{M}$ by a stationary approach instead of a time-dependent approach in [9].

\end{abstract}
\maketitle

\section{\bigskip Introduction}

This paper is a sequel to \cite{Li2}\ and \cite{Li}, in which we studied the
diagonalizations of self-adjoint operators modulo norm ideal in semifinite von
Neumann algebras. (see \cite{Kadison},\cite{M1}-\cite{M3} or \cite{V} for more
details about von Neumann algebras.) In particular, we give an analogue of
Kato-Rosenblum theorem in a semifinite von Neumann algebra in \cite{Li}.

Let $\mathcal{H}$ be a complex separable infinite dimensional Hilbert space.
Assume $H$ and $H_{1}$ are densely defined self-adjoint operators on
$\mathcal{H}$ satisfying that $H_{1}-H$ is in the trace class, then the
Kato-Rosenblum theorem asserts that the wave operator $W_{\pm}\left(
H_{1},H\right)  $ of $H$ and $H_{1}$ exists and consequently the absolutely
continuous parts of $H$ and $H_{1}$ are unitarily equivalent. Thus, if a
self-adjoint operator $H$ in $\mathcal{B}(\mathcal{H})$ has a nonzero
absolutely continuous spectrum, then $H$ can not be a sum of a diagonal
operator and a trace class operator. In \cite{Li}, we introduce the concept of
generalized wave operator $W_{\pm}$ based on the notion of norm absolutely
continuous projections. An analogue of Kato-Rosenblum theorem in a semifinite
von Neumann algebra $\mathcal{M}$ is obtained by showing the existence of the
generalized wave operator $W_{\pm}.$ To be more precise, we proved that for
self-adjoint operators $H$ and $H_{1}$ affiliated with $\mathcal{M}$
satisfying $H_{1}-H\in\mathcal{M\cap L}^{1}\left(  \mathcal{M},\tau\right)  ,$
the generalized wave operator $W_{\pm}\left(  H_{1},H\right)  $ exists and
then the norm absolutely continuous part of $H$ and $H_{1}$ are unitarily
equivalent. It implies that a self-adjoint operator $H$ affiliated with
$\mathcal{M}$ can not be a sum of a diagonal operator in $\mathcal{M}$ (see
Definition 1.0.1 in \cite{Li2}) and an operator in $\mathcal{M\cap L}%
^{1}\left(  \mathcal{M},\tau\right)  $ if it has a non-zero norm absolutely
continuous projection in $\mathcal{M}$

The above statements illustrate that showing the existence of wave operators
is the key step to prove two versions of Kato-Rosenblum theorem. In
mathematical scattering theory, wave operator $W_{\pm}$ is an elementary
concept and the existence of $W_{\pm}$ is one of the main research topics in
this area. Actually, there are two typical approaches to show the existence of
$W_{\pm}$. One is called time-dependent approach which has been used in
\cite{Kato} and \cite{Ro} and another is called stationary approach (see
\cite{BE}, \cite{D} or \cite{Y1}). The methods which do not make explicit use
of the time variable $t$ are known as the stationary approaches. An important
merit of a stationary approach is the advanced formula part. We notice that
the method in \cite{Li} to show the Kato-Rosenblum theorem in $\mathcal{M}$ is
a so-called time-dependent approach. So it is natural to ask whether there is
a stationary approach in $\mathcal{M}$. Thus to explore a stationary method in
$\mathcal{M}$ is our main purpose in the current article. We will also show
the Kato-Rosenblum theorem in $\mathcal{M}$ in \cite{Li} by a stationary approach.

The notion of the norm absolutely continuous support $P_{ac}^{\infty}\left(
H\right)  $ of a self-adjoint operator $H$ affiliated to $\mathcal{M}$ plays a
very important role in the Kato-Rosenblum theorem in $\mathcal{M}.$ So in this
article, we are going to characterize $P_{ac}^{\infty}\left(  H\right)  $ by
applying the Kato smoothness given in \cite{Kato2}, we assert that
\[
P_{ac}^{\infty}\left(  H\right)  =\vee\left\{  R\left(  G^{\ast}\right)
:G\in\mathcal{M}\text{ is }H\text{-smooth}\right\}  .
\]
Therefore for a self-adjoint $H$ affiliated with $\mathcal{M},$ if there is a
$H$-smooth operator in $\mathcal{M},$ then $H$ is not a sum of a diagonal
operator in $\mathcal{M}$ and an operator in $\mathcal{M\cap L}^{1}\left(
\mathcal{M},\tau\right)  $.

The construction of this paper is as follows. In section 2, we prepare related
notation, definitions and lemmas. We list the relation between the resolvent
$R_{H}\left(  z\right)  =\left(  H-z\right)  ^{-1}$ and unitary $U_{H}%
(t)=\exp\left(  -itH\right)  $ for a self-adjoint operator $H$ on
$\mathcal{H}.$ We also recall the definitions of Kato smoothness and
generalized wave operators. Some basic properties of generalized wave
operators are discussed in this section too. Section 3 is focused on the main
results of this paper. We first characterize the norm absolutely continuous
support $P_{ac}^{\infty}\left(  H\right)  $ of a self-adjoint operator $H$
affiliated to $\mathcal{M}$ by applying the Kato smoothness. After giving the
concepts of generalized weak wave operators $\widetilde{W}_{\pm}$, generalized
stationary wave operators $\mathcal{U}_{\pm}$ in $\mathcal{M}$, we give a
stationary proof of the Kato-Rosenblum theorem in $\mathcal{M}.$

\section{Preliminaries and Notation}

Let $\mathcal{H}$ be a complex Hilbert space and $\mathcal{B}\left(
\mathcal{H}\right)  $ be the set of all bounded linear operators on
$\mathcal{H}.$ In this article, we assume that $\mathcal{M}\subseteq
\mathcal{B}\left(  \mathcal{H}\right)  $ is a countable decomposable properly
infinite semifinite von Neumann algebra with a faithful normal tracial weight
$\tau$ and $\mathcal{A}\left(  \mathcal{M}\right)  $ is the set of densely
defined, closed operators affiliated with $\mathcal{M}.$

\subsection{The Unitary Group and Resolvent of a Self-adjoint Operator}

The resolvent $R_{H}\left(  z\right)  =\left(  H-z\right)  ^{-1}$ and unitary
$U_{H}(t)=\exp\left(  -itH\right)  $ for a self-adjoint operator $H$ on
$\mathcal{H}$ will be frequently used in the current paper$,$ so we recall
their properties and relations below.

Let $H$ be any self-adjoint operator with domain $\mathcal{D}\left(  H\right)
$ in $\mathcal{H}$ and $\left\{  \left(  E_{H}\left(  \lambda\right)  \right)
\right\}  _{\lambda\in\mathbb{R}}$ be the spectral resolution of the identity
for $H.$ For $f,g$ in $\mathcal{H},$ the unitary group $U_{H}(t)=\exp\left(
-itH\right)  $ has the sesquilinear form%
\begin{equation}
\left\langle U_{H}(t)f,g\right\rangle =\left\langle \exp\left(  -itH\right)
f,g\right\rangle =\int_{-\infty}^{\infty}\exp\left(  -i\lambda t\right)
d\left\langle E_{H}\left(  \lambda\right)  f,g\right\rangle . \label{g1}%
\end{equation}
Similarly, its resolvent $R_{H}\left(  z\right)  =\left(  H-z\right)  ^{-1}$
has the sesquilinear form
\[
\left\langle R_{H}\left(  z\right)  f,g\right\rangle =\int_{-\infty}^{\infty
}\left(  \lambda-z\right)  ^{-1}d\left\langle E_{H}\left(  \lambda\right)
f,g\right\rangle .
\]
The connection between above two sesquilinear forms is given by the relation
\begin{equation}
R_{H}\left(  \lambda\pm i\varepsilon\right)  =\pm i\int_{0}^{\infty}%
\exp\left(  -\varepsilon t\pm i\lambda t\right)  \exp\left(  \pm itH\right)
dt. \label{g3}%
\end{equation}
The proof of equality (\ref{g3}) is based on Fubini's Theorem and given in
Section 1.4 \cite{Y1}. Set
\[
\delta_{H}\left(  \lambda,\varepsilon\right)  =\frac{1}{2\pi i}\left[
R_{H}\left(  \lambda+i\varepsilon\right)  -R_{H}\left(  \lambda-i\varepsilon
\right)  \right]  =\frac{\varepsilon}{\pi}R_{H}\left(  \lambda+i\varepsilon
\right)  R_{H}\left(  \lambda-i\varepsilon\right)  \geq0,
\]
then
\begin{equation}
\varepsilon\pi^{-1}\left\Vert R_{H}\left(  \lambda\pm i\varepsilon\right)
f\right\Vert ^{2}=\left\langle \delta_{H}\left(  \lambda,\varepsilon\right)
f,f\right\rangle \label{g21}%
\end{equation}
and
\begin{equation}
\left\langle \delta_{H}\left(  \lambda,\varepsilon\right)  f,g\right\rangle
=\frac{\varepsilon}{\pi}\int_{-\infty}^{\infty}\frac{1}{\left(  s-\lambda
-i\varepsilon\right)  \left(  s-\lambda+i\varepsilon\right)  }d\left\langle
E_{H}\left(  s\right)  f,g\right\rangle . \label{g40}%
\end{equation}
Denote by $\mathcal{H}_{ac}\left(  H\right)  $ the set of all these vectors
$x\in\mathcal{H}$ such that the mapping $\lambda\longmapsto\left\langle
E_{H}\left(  \lambda\right)  x,x\right\rangle ,$ with $\lambda\in\mathbb{R}$,
is a locally absolutely continuous function on $\mathbb{R}$ (see \cite{Kato4}
or \cite{Y1} for more details). From the argument in Section 1.4 \cite{Y1}, we
conclude that
\begin{equation}
\lim_{\varepsilon\rightarrow0}\left\langle \delta_{H}\left(  \lambda
,\varepsilon\right)  f,g\right\rangle =\frac{d\left\langle E_{H}\left(
\lambda\right)  f,g\right\rangle }{d\lambda},\text{ \ \ a.e. }\lambda
\in\mathbb{R} \label{g6}%
\end{equation}
for $f$ or $g$ in $\mathcal{H}_{ac}\left(  H\right)  $. We also have
\begin{equation}
\frac{d\left\langle E_{H}\left(  \lambda\right)  E_{H}\left(  \Lambda\right)
f,g\right\rangle }{d\lambda}=\mathcal{X}_{\Lambda}\left(  \lambda\right)
\frac{d\left\langle E_{H}\left(  \lambda\right)  f,g\right\rangle }{d\lambda
},\text{ a.e. }\lambda\in\mathbb{R}, \label{g30}%
\end{equation}
where $\mathcal{X}_{\Lambda}\left(  \cdot\right)  $ is the characteristic
function of the Borel set $\Lambda$. The proof of equality (\ref{g30}) can be
found in the proof of Theorem X.4.4 in \cite{Kato4} or Section 1.3 in
\cite{Y1}.

\subsection{Kato Smoothness and generalized wave operators}

Kato smoothness play a very important role in the mathematical scattering
theory. It can be equivalently formulated in terms of the corresponding
unitary group. We recall it in this section.

For a self-adjoint operator $H,$ an operator $G:\mathcal{H\rightarrow H}$ is
called $H$-bounded if $\mathcal{D}\left(  H\right)  \subseteq\mathcal{D}%
\left(  G\right)  $ and $GR_{H}\left(  z\right)  $ is bounded for $z$ in the
resolvent set $\rho=\rho\left(  H\right)  .$

\begin{theorem}
\label{K1}(Theorem 4.3.1 in \cite{Y1} or Theorem 5.1 in \cite{Kato2}) Let $H$
be a densely defined self-adjoint operator in $\mathcal{H}$. Assume that
$G:\mathcal{H\rightarrow H}$ is $H$-bounded operator$,$ then the following
conditions are equivalent.

\begin{enumerate}
\item $\gamma_{1}^{2}=\frac{1}{2\pi}\sup_{f\in\mathcal{D}\left(  H\right)
,\left\Vert f\right\Vert =1}\int_{\mathbb{R}}\left\Vert Ge^{\pm itH}%
f\right\Vert ^{2}dt<\infty;$

\item $\gamma_{2}^{2}=\frac{1}{\left(  2\pi\right)  ^{2}}\sup_{\left\Vert
f\right\Vert =1,\varepsilon>0}\int_{\mathbb{R}}\left(  \left\Vert
GR_{H}(\lambda+i\varepsilon)f\right\Vert ^{2}+\left\Vert GR_{H}\left(
\lambda-i\varepsilon\right)  f\right\Vert ^{2}d\lambda\right)  <\infty;$

\item $\gamma_{3}^{2}=\sup_{\left\Vert f\right\Vert =1,\varepsilon>0}%
\int_{\mathbb{R}}\left\Vert G\delta_{H}\left(  \lambda,\varepsilon\right)
f\right\Vert ^{2}d\lambda<\infty;$

\item $\gamma_{4}^{2}=\sup_{\lambda\in\mathbb{R},\varepsilon>0}\left\Vert
G\delta_{H}\left(  \lambda,\varepsilon\right)  G^{\ast}\right\Vert <\infty;$

\item $\gamma_{5}^{2}=\sup_{\Lambda\subseteq\mathbb{R}}\frac{\left\Vert
GE_{H}\left(  \Lambda\right)  G^{\ast}\right\Vert }{\left\vert \Lambda
\right\vert }<\infty.$
\end{enumerate}

All the constants $\gamma_{j}=\gamma_{j}\left(  G\right)  ,$ $j=1,\cdots,5,$
are equal to one another.
\end{theorem}

\begin{definition}
\label{K2}Let $H$ be a self-adjoint operator acting on the Hilbert space
$\mathcal{H}.$ If $G$ is $H$-bounded and one of the inequalities (1)-(5) holds
(and then all of them), then operator $G$ is called Kato smooth relative to
the operator $H$ ($H$-smooth). The common value of the quantities $\gamma
_{1},\cdots,\gamma_{5}$ is denoted by $\gamma_{H}\left(  G\right)  .$
\end{definition}

\begin{remark}
There are other expressions for the number $\gamma_{H}\left(  G\right)  $
given in the Section 4.3 (\cite{Y1}). In particular, for each the sign "$\pm
$"
\begin{equation}
\gamma_{H}^{2}\left(  G\right)  =\left(  \frac{1}{2\pi}\right)  ^{2}%
\sup_{\left\Vert f\right\Vert =1,\varepsilon>0}\int_{-\infty}^{\infty
}\left\Vert GR_{H}\left(  \lambda\pm i\varepsilon\right)  f\right\Vert
^{2}d\lambda\label{g7}%
\end{equation}

\end{remark}

Before giving the definition of generalized wave operators in $\mathcal{M},$
we need to recall the following concepts which appear first in \cite{Li}.

\begin{definition}
\label{M0}(\cite{Li})Let $H$ be a self-adjoint element in $\mathcal{A}\left(
\mathcal{M}\right)  $ and let $\left\{  E_{H}\left(  \lambda\right)  \right\}
_{\lambda\in\mathbb{R}}$ be the spectral resolution of the identity for $H$ in
$\mathcal{M}.$ We define $\mathcal{P}_{ac}^{\infty}\left(  H\right)  $ to be
the collection of those projections $P$ in $\mathcal{M}$ such that:

the mapping $\lambda\longmapsto PE_{H}\left(  \lambda\right)  P$ from
$\lambda\in\mathbb{R}$ into $\mathcal{M}$ is locally absolutely continuous,
i.e., for all $a,b\in\mathbb{R}$ with $a<b$ and every $\varepsilon>0,$ there
exists a $\delta>0$ such that $\sum_{i}\left\Vert PE_{H}\left(  b_{i}\right)
P-PE_{H}\left(  a_{i}\right)  P\right\Vert <\varepsilon$ for every finite
collection $\left\{  \left(  a_{i},b_{i}\right)  \right\}  $ of disjoint
intervals in $\left[  a,b\right]  $ with $\sum_{i}\left(  b_{i}-a_{i}\right)
<\delta.$

A projection $P\in\mathcal{P}_{ac}^{\infty}\left(  H\right)  $ is called a
norm absolutely continuous projection with respect to $H.$ Define
\[
P_{ac}^{\infty}\left(  H\right)  =\vee\left\{  P:P\in\mathcal{P}_{ac}^{\infty
}\left(  H\right)  \right\}  .
\]
Such $P_{ac}^{\infty}\left(  H\right)  $ is called the norm absolutely
continuous support of $H$ in $\mathcal{M}$ and denote the range of
$P_{ac}^{\infty}\left(  H\right)  $ by $\mathcal{H}_{ac}^{\infty}\left(
H\right)  .$
\end{definition}

\begin{remark}
\label{I1}Let $P_{ac}\left(  H\right)  $ be the projection from $\mathcal{H}$
onto $\mathcal{H}_{ac}\left(  H\right)  .$ In \cite{Li}, it has been shown
that $P_{ac}^{\infty}\left(  H\right)  \leq P_{ac}\left(  H\right)  $ and
$P_{ac}^{\infty}\left(  H\right)  \in\mathcal{M\cap A}^{\prime}$ where
$\mathcal{A}$ is the von Neumann subalgebra generated by $\left\{
E_{H}\left(  \lambda\right)  \right\}  _{\lambda\in\mathbb{R}}$ in
$\mathcal{M}$ and $\mathcal{A}^{\prime}$ denotes the commutant of
$\mathcal{A}.$
\end{remark}

Now, we are ready to recall the definition of generalized wave operators.

\begin{definition}
\label{I2}(\cite{Li})Let $H$, $H_{1}\in\mathcal{A}\left(  \mathcal{M}\right)
$ be a pair of self-adjoint operators and $J$ be an operator in $\mathcal{M}$.
The generalized wave operator for a pair of self-adjoint operators $H$,
$H_{1}$ and $J$ in $\mathcal{M}$ is the operator
\[
W_{\pm}(H_{1},H;J)=s.o.t\text{-}\lim_{t\rightarrow\pm\infty}e^{itH_{1}%
}Je^{-itH}P_{ac}^{\infty}\left(  H\right)
\]
provided that $s.o.t$ (strong operator topology) limit exists.
\end{definition}

We note that the relation containing the signs "$\pm$" is understood as two
independent equalities. After slightly modify the proof of Theorem 5.2.5 in
\cite{Li}, we can get the next result.

\begin{theorem}
\label{I3}Let $H$, $H_{1}\in\mathcal{A}\left(  \mathcal{M}\right)  $ be a pair
of self-adjoint operators and $J$ be an operator in $\mathcal{M}$. If $W_{\pm
}(H_{1},H;J)$ exists for a pair of self-adjoint operators $H$, $H_{1}$ and
$J,$ then for any Borel function $\varphi$%
\[
\varphi\left(  H_{1}\right)  W_{\pm}(H_{1},H;J)=W_{\pm}(H_{1},H;J)\varphi
\left(  H\right)  .
\]
In particular, for any Borel set $\Lambda\subseteq\mathbb{R}$
\[
E_{H_{1}}\left(  \Lambda\right)  W_{\pm}(H_{1},H;J)=W_{\pm}(H_{1}%
,H;J)E_{H}\left(  \Lambda\right)  .
\]

\end{theorem}

Different $J$ might give us different $W_{\pm},$ we will give a condition
about $J$ such that $W_{\pm}(H_{1},H;J)$ is an isometry on $P_{ac}^{\infty
}\left(  H\right)  $ below$.$ Its proof is similar to Proposition 2.1.3 in
\cite{Y1}, so we omit it.

\begin{theorem}
\label{I4}Let $H$, $H_{1}\in\mathcal{A}\left(  \mathcal{M}\right)  $ be a pair
of self-adjoint operators and $J$ be an operator in $\mathcal{M}$. If $W_{\pm
}(H_{1},H;J)$ exists, then $W_{\pm}(H_{1},H;J)$ is isometric on $P_{ac}%
^{\infty}\left(  H\right)  $ if the strong operator limit
\[
s.o.t\text{-}\lim_{t\rightarrow\pm\infty}\left(  J^{\ast}J-I\right)
e^{-itH}P_{ac}^{\infty}\left(  H\right)  =0
\]

\end{theorem}

\begin{lemma}
\label{M6}(Lemma 5.2.3 \cite{Li}) Suppose $H$ is a self-adjoint element in
$\mathcal{A}\left(  \mathcal{M}\right)  .$ Let $\left\{  \left(  E_{H}\left(
\lambda\right)  \right)  \right\}  _{\lambda\in\mathbb{R}}$ be the spectral
resolution of the identity for $H$ in $\mathcal{M}.$ If $S\in\mathcal{M}$
satisfies that the mapping $\lambda\longmapsto S^{\ast}E_{H}\left(
\lambda\right)  S$ from $\mathbb{R}$ into $\mathcal{M}$ is locally absolutely
continuous, then $R\left(  S\right)  ,$ the range projection of $S$ in
$\mathcal{M},$ is a subprojection of $P_{ac}^{\infty}\left(  H\right)  .$
\end{lemma}

Then we can get the following result.

\begin{proposition}
Let $H$, $H_{1}\in\mathcal{A}\left(  \mathcal{M}\right)  $ be a pair of
self-adjoint operators and $J$ be an operator in $\mathcal{M}$. If $W_{\pm
}\overset{\triangle}{=}W_{\pm}(H_{1},H;J)$ exists and the strong operator
limit
\[
s.o.t\text{-}\lim_{t\rightarrow\pm\infty}\left(  J^{\ast}J-I\right)
e^{-itH}P_{ac}^{\infty}\left(  H\right)  =0,
\]

Then $W_{\pm}W_{\pm}^{\ast}\leq P_{ac}^{\mathcal{1}}\left(  H_{1}\right)  $
\end{proposition}

\begin{proof}
If
\[
s.o.t\text{-}\lim_{t\rightarrow\pm\infty}\left(  J^{\ast}J-I\right)
e^{-itH}P_{ac}^{\infty}\left(  H\right)  =0,
\]
then $W_{\pm}^{\ast}W_{\pm}=P_{ac}^{\infty}\left(  H\right)  $ by Theorem
\ref{I4}. From Theorem \ref{I3}, for any $P\in\mathcal{P}_{ac}^{\mathcal{1}%
}\left(  H\right)  $ and any Borel set $\Lambda\subseteq\mathbb{R},$
\[
\left(  W_{\pm}P\right)  ^{\ast}E_{H_{1}}\left(  \Lambda\right)  \left(
W_{\pm}P\right)  =PW_{\pm}^{\ast}E_{H_{1}}\left(  \Lambda\right)  W_{\pm
}P=PW_{\pm}^{\ast}W_{\pm}E_{H}\left(  \Lambda\right)  P=PE_{H}\left(
\Lambda\right)  P.
\]
It implies that the mapping $\lambda\rightarrow\left(  W_{\pm}P\right)
^{\ast}E_{H_{1}}\left(  \lambda\right)  \left(  W_{\pm}P\right)  $ from
$\mathbb{R}$ into $\mathcal{M}$ is locally absolutely continuous. Hence the
range projection $R\left(  W_{\pm}P\right)  \leq P_{ac}^{\mathcal{1}}\left(
H_{1}\right)  $ by Lemma \ref{M6}$.$ Therefore $R\left(  W_{\pm}\right)  \leq
P_{ac}^{\mathcal{1}}\left(  H_{1}\right)  $ by the fact that $W_{\pm}%
P_{ac}^{\mathcal{1}}\left(  H\right)  =W_{\pm}.$ Hence $W_{\pm}W_{\pm}^{\ast
}\leq P_{ac}^{\mathcal{1}}\left(  H_{1}\right)  .$
\end{proof}

\section{\bigskip Main Results}

\subsection{Characterization of norm absolutely continuous projections in
$\mathcal{M}$}

The cut off function $\omega_{n}$ is given in \cite{Li}. We refer the reader
to \cite{Li} for its definition. Here we only recall its useful property.

\begin{lemma}
\label{M5}(Lemma 4.2.2 in \cite{Li}) Suppose $H$ is a self-adjoint element in
$\mathcal{A}\left(  \mathcal{M}\right)  .$ For each $n\in\mathbb{N}$ and
cut-off function $\omega_{n},$ let
\[
\omega_{n}\left(  H\right)  =\int_{\mathbb{R}}\omega_{n}\left(  t\right)
dE_{H}\left(  t\right)  .
\]
Then $\omega_{n}\left(  H\right)  \in\mathcal{M}$ and%
\[
\omega_{n}\left(  H\right)  \rightarrow I\text{ in strong operator topology,
as }n\rightarrow\mathcal{1}.
\]

\end{lemma}

\begin{remark}
\label{M26} Let $H$ be a self-adjoint element in $\mathcal{A}\left(
\mathcal{M}\right)  $ and $P\in\mathcal{P}_{ac}^{\infty}\left(  H\right)  .$
From Lemma 4.2.3 (vi) in \cite{Li},
\begin{equation}
\int_{\mathbb{R}}\left\Vert P\omega_{n}\left(  H\right)  e^{-itH}f\right\Vert
^{2}dt\leq\frac{n}{2\pi}\left\Vert f\right\Vert \label{g70}%
\end{equation}
for any $f\in\mathcal{H}$ and $n\in\mathbb{N}.$ Then%
\[
\sup_{\left\Vert f\right\Vert =1}\int_{\mathbb{R}}\left\Vert P\omega
_{n}\left(  H\right)  e^{-itH}f\right\Vert ^{2}dt\leq\frac{n}{2\pi}.
\]
By Theorem \ref{K1} and Definition \ref{K2}, we have $G=P\omega_{n}\left(
H\right)  $ is $H$-smooth for $n\in\mathbb{N}.$
\end{remark}

Next theorem is the main result in this subsection.

\begin{theorem}
\label{M7}Suppose $H$ is a self-adjoint element in $\mathcal{A}\left(
\mathcal{M}\right)  .$ Then
\[
P_{ac}^{\infty}\left(  H\right)  =\vee\left\{  R\left(  G^{\ast}\right)
:G\in\mathcal{M}\text{ is }H\text{-smooth}\right\}
\]
where $R\left(  G^{\ast}\right)  $ is the range projection of $G^{\ast}.$
\end{theorem}

\begin{proof}
By Remark \ref{M26}, we have $G=P\omega_{n}\left(  H\right)  $ is $H$-smooth.
Hence from Theorem \ref{K1},
\[
\sup_{\left\Vert x\right\Vert =1}\frac{1}{2\pi}\int_{\mathbb{R}}\left\Vert
P\omega_{n}\left(  H\right)  e^{-itH}x\right\Vert ^{2}dt=\sup_{\Lambda
\subseteq\mathbb{R}}\frac{\left\Vert P\omega_{n}\left(  H\right)  E_{H}\left(
\Lambda\right)  \omega_{n}\left(  H\right)  P\right\Vert }{\left\vert
\Lambda\right\vert }\leq\frac{n}{\left(  2\pi\right)  ^{2}}.
\]
Therefore $\lambda\rightarrow P\omega_{n}\left(  H\right)  E_{H}\left(
\lambda\right)  \omega_{n}\left(  H\right)  P$ is locally absolutely
continuous$.$ Then by Lemma \ref{M6}, we have $R\left(  \omega_{n}\left(
H\right)  P\right)  \leq P_{ac}^{\infty}\left(  H\right)  $ for every
$n\in\mathbb{N}$ and $P\in\mathcal{P}_{ac}^{\infty}\left(  H\right)  .$ Hence
\[
P=\vee_{n}R\left(  \omega_{n}\left(  H\right)  P\right)  =\vee_{n}\left\{
R\left(  G^{\ast}\right)  :G=P\omega_{n}\left(  H\right)  \right\}
\]
by Lemma \ref{M5}, now we conclude that $P_{ac}^{\infty}\left(  H\right)
\leq\vee\left\{  R\left(  G^{\ast}\right)  :G\in\mathcal{M}\text{ is
}H\text{-smooth}\right\}  .$

On the other hand, if $G\in\mathcal{M}$ is $H$-smooth, then by Theorem
\ref{K1} (5) we have $\lambda\rightarrow GE_{\lambda}G^{\ast}$ is locally
absolutely continuous. Therefore $R\left(  G^{\ast}\right)  \leq
P_{a.c}^{\infty}\left(  H\right)  $ by Lemma \ref{M6}. Hence
\[
\vee\left\{  R\left(  G^{\ast}\right)  :G\text{ is }H\text{-smooth in
}\mathcal{M}\right\}  \leq P_{a.c}^{\infty}\left(  H\right)  .
\]
This completes the proof.
\end{proof}

Let $P_{ac}\left(  H\right)  $ be the projection from $\mathcal{H}$ onto
$\mathcal{H}_{ac}\left(  H\right)  .$ In \cite{Li}, it has been shown that
$P_{ac}\left(  H\right)  =P_{ac}^{\infty}\left(  H\right)  $ for a densely
defined self-adjoint operator $H$ (therefore $H\in\mathcal{A}\left(
\mathcal{B}\left(  \mathcal{H}\right)  \right)  .$ Then we can get the
following corollary.

\begin{corollary}
\label{M8}Let $H$ be a densely defined self-adjoint operator on $\mathcal{H}.$
Then
\begin{align*}
P_{ac}\left(  H\right)   &  =P_{ac}^{\infty}\left(  H\right) \\
&  =\vee\left\{  R\left(  G^{\ast}\right)  :G\in\mathcal{B}\left(
\mathcal{H}\right)  \text{ is }H\text{-smooth}\right\}  .
\end{align*}

\end{corollary}

\begin{corollary}
\label{M9}Suppose $H$ is a self-adjoint affiliated with $\mathcal{M}.$ Then
$P_{ac}^{\infty}\left(  H\right)  \neq0$ if and only if there is at least one
$H$-smooth operator in $\mathcal{M}.$
\end{corollary}

\begin{proof}
If $P_{ac}^{\infty}\left(  H\right)  \neq0,$ then there is a projection
$P\in\mathcal{P}_{ac}^{\infty}\left(  H\right)  .$ By the argument in the
proof of Theorem \ref{M7}, we know that $P\omega_{n}\left(  H\right)  $ is
$H$-smooth operator in $\mathcal{M}.$ The other direction is clear by Theorem
\ref{M7}.
\end{proof}

\subsection{A Stationary approach in $\mathcal{M}$}

\bigskip\ \ The stationary approach in a Hilbert space $\mathcal{H}$ is based
on several variations of wave operators, such as weak wave operators and
stationary wave operators (see \cite{Y1}). So we will also give the
definitions of these variations in $\mathcal{M}.$ For the reader who is
familiar with the general scattering theory, the following definitions are
natural extensions of the corresponding definitions in the general scattering
theory. For the reader who is not familiar with this area, we refer the reader
to the Appendix of this paper for the details.

\begin{definition}
\label{R1} Let $H$, $H_{1}\in\mathcal{A}\left(  \mathcal{M}\right)  $ be a
pair of self-adjoint operators and $J$ be an operator in $\mathcal{M}$. The
generalized weak wave operator for a pair of self-adjoint operators $H$,
$H_{1}$ and $J$ is the operator
\[
\widetilde{W}_{\pm}(H_{1},H;J)=w.o.t\text{-}\lim_{t\rightarrow\pm\infty}%
P_{ac}^{\infty}\left(  H_{1}\right)  e^{itH_{1}}Je^{-itH}P_{ac}^{\infty
}\left(  H\right)
\]
provided that $w.o.t$ (weak operator topology) limit exists.
\end{definition}

\ Furthermore, we also have
\begin{equation}
\widetilde{W}_{\pm}(H,H_{1};J^{\ast})=\widetilde{W}_{\pm}^{\ast}(H_{1},H;J)
\label{g33}%
\end{equation}
if $\widetilde{W}_{\pm}(H_{1},H;J)$ exists.

\begin{definition}
\label{R5}Let $J$ be an operator in $\mathcal{M}$, $H$, $H_{1}$ be
self-adjoint operators in $\mathcal{A}\left(  \mathcal{M}\right)  .$ If for
any pair of elements $f$ and $f_{1}$ in $\mathcal{H},$
\[
\lim_{\varepsilon\rightarrow0}\frac{\varepsilon}{\pi}\left\langle
JR_{H}\left(  \lambda\pm i\varepsilon\right)  P_{ac}^{\infty}\left(  H\right)
f,R_{H_{1}}\left(  \lambda\pm i\varepsilon\right)  P_{ac}^{\infty}\left(
H_{1}\right)  f_{1}\right\rangle
\]
exists for a.e. $\lambda\in\mathbb{R},$ then the generalized stationary wave
operator is defined as
\[
\left\langle \mathcal{U}_{\pm}\left(  H_{1},H;J\right)  f,f_{1}\right\rangle
=\int_{-\infty}^{\infty}\lim_{\varepsilon\rightarrow0}\frac{\varepsilon}{\pi
}\left\langle JR_{H}\left(  \lambda\pm i\varepsilon\right)  P_{ac}^{\infty
}\left(  H\right)  f,R_{H_{1}}\left(  \lambda\pm i\varepsilon\right)
P_{ac}^{\infty}\left(  H_{1}\right)  f_{1}\right\rangle d\lambda.
\]

\end{definition}

From the definition of $\mathcal{U}_{\pm}\left(  H_{1},H;J\right)  $, it is
clear that
\begin{equation}
P_{ac}^{\infty}\left(  H_{1}\right)  \mathcal{U}_{\pm}\left(  H_{1}%
,H;J\right)  =\mathcal{U}_{\pm}\left(  H_{1},H;J\right)  \label{G70}%
\end{equation}
if $\mathcal{U}_{\pm}\left(  H_{1},H;J\right)  $ exists.

Note $\mathcal{U}_{\pm}\left(  H_{1},H;J\right)  $ is given in terms of the
resolvents of operators $H$ and $H_{1}$ which obviously have nothing to do
with time variable $t$ apparently$.$ So $\mathcal{U}_{\pm}\left(
H_{1},H;J\right)  $ is a key concept in the stationary approach in
$\mathcal{M}.$ Actually, to check the existence of $\mathcal{U}_{\pm}\left(
H_{1},H;J\right)  $ is one of the key steps to show the existence of $W_{\pm}$
in a stationary method.

\begin{corollary}
\label{R8}Let $H$, $H_{1}\in\mathcal{A}\left(  \mathcal{M}\right)  $ be a pair
of self-adjoint operators and $J$ be an operator in $\mathcal{M}$. If
$\mathcal{U}_{\pm}\left(  H_{1},H;J\right)  $ exists, then
\[
\mathcal{U}_{\pm}\left(  H,H_{1};J^{\ast}\right)  =\mathcal{U}_{\pm}^{\ast
}\left(  H_{1},H;J\right)  .
\]

\end{corollary}

Next result give us the relation among $\mathcal{U}_{\pm}\left(
H_{1},H;J\right)  ,$ $\widetilde{W}_{\pm}\left(  H_{1},H;J\right)  $ and
$W_{\pm}(H_{1},H;J)$, it is a natural extension of the similar result in the
general scattering theory and the proof are similar too. If the reader are
interested in its proof, you can find it in the Appendix of this paper.

\begin{theorem}
\label{R10}Let $H$, $H_{1}\in\mathcal{A}\left(  \mathcal{M}\right)  $ be a
pair of self-adjoint operators and $J$ be an operator in $\mathcal{M}$. If
$\mathcal{U}_{\pm}\left(  H_{1},H;J\right)  $, $\mathcal{U}_{\pm}\left(
H,H;J^{\ast}J\right)  $, $\widetilde{W}_{\pm}\left(  H_{1},H;J\right)  $ and
$\widetilde{W}_{\pm}\left(  H,H;J^{\ast}J\right)  $ exist as well as
\[
\mathcal{U}_{\pm}^{\ast}\left(  H_{1},H;J\right)  \mathcal{U}_{\pm}\left(
H_{1},H;J\right)  =\mathcal{U}_{\pm}\left(  H,H;J^{\ast}J\right)  ,
\]
then $W_{\pm}(H_{1},H;J)$ exists and%
\[
\mathcal{U}_{\pm}\left(  H_{1},H;J\right)  =W_{\pm}(H_{1},H;J).
\]

\end{theorem}

Now based on the definition of $\mathcal{U}_{\pm}\left(  H_{1},H;J\right)  ,$
we can get the following property.

\begin{corollary}
\label{R9}Let $H$, $H_{1}\in\mathcal{A}\left(  \mathcal{M}\right)  $ be a pair
of self-adjoint operators and $J$ be an operator in $\mathcal{M}$. If
$\mathcal{U}_{\pm}\left(  H_{1},H;J\right)  $ exists, then for any pair of
elements $f$ and $f_{1}$ in $\mathcal{H}\ $and any Borel sets $\Lambda$,
$\Lambda_{1}\subset\mathbb{R},$
\begin{align*}
&  \left\langle \mathcal{U}_{\pm}\left(  H_{1},H;J\right)  E_{H}\left(
\Lambda\right)  f,E_{H_{1}}\left(  \Lambda_{1}\right)  f_{1}\right\rangle \\
&  =\int_{\Lambda_{1}\cap\Lambda}\lim_{\varepsilon\rightarrow0}\frac
{\varepsilon}{\pi}\left\langle JR_{H}\left(  \lambda\pm i\varepsilon\right)
P_{ac}^{\infty}\left(  H\right)  f,R_{H_{1}}\left(  \lambda\pm i\varepsilon
\right)  P_{ac}^{\infty}\left(  H_{1}\right)  f_{1}\right\rangle d\lambda.
\end{align*}

\end{corollary}

\begin{proof}
Set
\[
\alpha_{\pm}\left(  f,f_{1};\lambda\right)  =\lim_{\varepsilon\rightarrow
0}\frac{\varepsilon}{\pi}\left\langle JR_{H}\left(  \lambda\pm i\varepsilon
\right)  P_{ac}^{\infty}\left(  H\right)  f,R_{H_{1}}\left(  \lambda\pm
i\varepsilon\right)  P_{ac}^{\infty}\left(  H_{1}\right)  f_{1}\right\rangle
.
\]
Since
\[
P_{ac}^{\infty}\left(  H\right)  E_{H}\left(  \Lambda\right)  =E_{H}\left(
\Lambda\right)  P_{ac}^{\infty}\left(  H\right)
\]
and
\[
P_{ac}^{\infty}\left(  H_{1}\right)  E_{H_{1}}\left(  \Lambda_{1}\right)
=E_{H_{1}}\left(  \Lambda_{1}\right)  P_{ac}^{\infty}\left(  H_{1}\right)
\]
by Remark \ref{I1}, we get that
\begin{align*}
&  \left\vert \alpha_{\pm}\left(  E_{H}\left(  \Lambda\right)  f,E_{H_{1}%
}\left(  \Lambda_{1}\right)  f_{1};\lambda\right)  \right\vert ^{2}\\
&  \leq\frac{\varepsilon^{2}}{\pi^{2}}\left\Vert J\right\Vert ^{2}%
\lim_{\varepsilon\rightarrow0}\left\Vert R_{H}\left(  \lambda\pm
i\varepsilon\right)  E_{H}\left(  \Lambda\right)  P_{ac}^{\infty}\left(
H\right)  f\right\Vert ^{2}\lim_{\varepsilon\rightarrow0}\left\Vert R_{H_{1}%
}\left(  \lambda\pm i\varepsilon\right)  E_{H_{1}}\left(  \Lambda_{1}\right)
P_{ac}^{\infty}\left(  H_{1}\right)  f_{1}\right\Vert ^{2}\\
&  =\left\Vert J\right\Vert ^{2}\lim_{\varepsilon\rightarrow0}\left\langle
\delta_{H}\left(  \lambda,\varepsilon\right)  E_{H}\left(  \Lambda\right)
P_{ac}^{\infty}\left(  H\right)  f,f\right\rangle \cdot\lim_{\varepsilon
\rightarrow0}\left\langle \delta_{H_{1}}\left(  \lambda,\varepsilon\right)
E_{H_{1}}\left(  \Lambda_{1}\right)  P_{ac}^{\infty}\left(  H_{1}\right)
f_{1},f_{1}\right\rangle \\
&  =\left\Vert J\right\Vert ^{2}\mathcal{X}_{\Lambda\cap\Lambda_{1}}%
\frac{d\left\langle E_{H}\left(  \lambda\right)  P_{ac}^{\infty}\left(
H\right)  f,f\right\rangle }{d\lambda}\frac{d\left\langle E_{H_{1}}\left(
\lambda\right)  P_{ac}^{\infty}\left(  H_{1}\right)  f_{1},f_{1}\right\rangle
}{d\lambda}%
\end{align*}
by (\ref{g6}) and (\ref{g30}) where $\mathcal{X}_{\Lambda\cap\Lambda_{1}}$ is
the characteristic function of $\Lambda\cap\Lambda_{1}.$ Therefore
\[
\mathcal{X}_{\mathbb{R}\backslash\Lambda\cap\Lambda_{1}}\alpha_{\pm}\left(
E_{H}\left(  \Lambda\right)  f,E_{H_{1}}\left(  \Lambda_{1}\right)
f_{1};\lambda\right)  =0\text{.}%
\]
where $\mathcal{X}_{\mathbb{R}\backslash\Lambda\cap\Lambda_{1}}$ is the
characteristic function of $\mathbb{R}\backslash\Lambda\cap\Lambda_{1}.$ It
implies that
\[
\mathcal{X}_{\Lambda\cap\Lambda_{1}}\alpha_{\pm}\left(  E_{H}\left(
\Lambda\right)  f,E_{H_{1}}\left(  \Lambda_{1}\right)  f_{1};\lambda\right)
=\alpha_{\pm}\left(  E_{H}\left(  \Lambda\right)  f,E_{H_{1}}\left(
\Lambda_{1}\right)  f_{1};\lambda\right)  .
\]
Hence
\begin{align*}
&  \mathcal{X}_{\Lambda\cap\Lambda_{1}}\alpha_{\pm}\left(  f,f_{1}%
;\lambda\right) \\
&  =\mathcal{X}_{\Lambda\cap\Lambda_{1}}\alpha_{\pm}\left(  E_{H}\left(
\Lambda\right)  f,E_{H_{1}}\left(  \Lambda_{1}\right)  f_{1};\lambda\right)
+\mathcal{X}_{\Lambda\cap\Lambda_{1}}\alpha_{\pm}\left(  E_{H}\left(
\mathbb{R}\backslash\left(  \Lambda\right)  \right)  f,E_{H_{1}}\left(
\Lambda_{1}\right)  f_{1};\lambda\right) \\
&  +\mathcal{X}_{\Lambda\cap\Lambda_{1}}\alpha_{\pm}\left(  f,E_{H_{1}}\left(
\mathbb{R}\backslash\left(  \Lambda_{1}\right)  \right)  f_{1};\lambda\right)
\\
&  =\alpha_{\pm}\left(  E_{H}\left(  \Lambda\right)  f,E_{H_{1}}\left(
\Lambda_{1}\right)  f_{1};\lambda\right)  .
\end{align*}
It follows that
\begin{align*}
&  \left\langle \mathcal{U}_{\pm}\left(  H_{1},H;J\right)  E_{H}\left(
\Lambda\right)  f,E_{H_{1}}\left(  \Lambda_{1}\right)  f_{1}\right\rangle \\
&  =\int_{\mathbb{R}}\alpha_{\pm}\left(  E_{H}\left(  \Lambda\right)
f,E_{H_{1}}\left(  \Lambda_{1}\right)  f_{1};\lambda\right)  =\int
_{\mathbb{R}}\mathcal{X}_{\Lambda\cap\Lambda_{1}}\alpha_{\pm}\left(
f,f_{1};\lambda\right) \\
&  =\int_{\Lambda_{1}\cap\Lambda}\lim_{\varepsilon\rightarrow0}\frac
{\varepsilon}{\pi}\left\langle JR_{H}\left(  \lambda\pm i\varepsilon\right)
P_{ac}^{\infty}\left(  H\right)  f,R_{H_{1}}\left(  \lambda\pm i\varepsilon
\right)  P_{ac}^{\infty}\left(  H_{1}\right)  f_{1}\right\rangle d\lambda.
\end{align*}
The proof is completed.
\end{proof}

\begin{remark}
\label{M10}Let $H$, $H_{1}\in\mathcal{A}\left(  \mathcal{M}\right)  $ be a
pair of self-adjoint operators and $J$ be an operator in $\mathcal{M}$ with
$J\mathcal{D}\left(  H\right)  \subseteq\mathcal{D}\left(  H_{1}\right)  $.
Then
\begin{align}
H_{1}J-JH  &  =\left(  H_{1}-z\right)  J-J\left(  H-z\right) \nonumber\\
&  =\left(  H_{1}-z\right)  (JR_{H}\left(  z\right)  -R_{H_{1}}\left(
z\right)  J)\left(  H-z\right)  . \label{g71}%
\end{align}
Hence
\begin{equation}
JR_{H}\left(  z\right)  -R_{H_{1}}\left(  z\right)  J=R_{H_{1}}\left(
z\right)  \left(  H_{1}J-JH\right)  R_{H}\left(  z\right)  . \label{g28}%
\end{equation}
From (\ref{g21}) and (\ref{g28}), we have
\begin{align}
&  \frac{\varepsilon}{\pi}\left\langle JR_{H}\left(  \lambda\pm i\varepsilon
\right)  P_{ac}^{\infty}\left(  H\right)  f,R_{H_{1}}\left(  \lambda\pm
i\varepsilon\right)  P_{ac}^{\infty}\left(  H_{1}\right)  f_{1}\right\rangle
\nonumber\\
&  =\frac{\varepsilon}{\pi}\left\langle \left(  R_{H_{1}}\left(  \lambda\pm
i\varepsilon\right)  J+R_{H_{1}}\left(  \lambda\pm i\varepsilon\right)
\left(  H_{1}J-JH\right)  R_{H}\left(  \lambda\pm i\varepsilon\right)
\right)  P_{ac}^{\infty}\left(  H\right)  f,R_{H_{1}}\left(  \lambda\pm
i\varepsilon\right)  P_{ac}^{\infty}\left(  H_{1}\right)  f_{1}\right\rangle
\nonumber\\
&  =\left\langle \left(  J+\left(  H_{1}J-JH\right)  R_{H}\left(  \lambda\pm
i\varepsilon\right)  \right)  P_{ac}^{\infty}\left(  H\right)  f,\delta
_{H_{1}}\left(  \lambda,\varepsilon\right)  P_{ac}^{\infty}\left(
H_{1}\right)  f_{1}\right\rangle . \label{g26}%
\end{align}

\end{remark}

\begin{lemma}
\label{M11}Let $H$, $H_{1}\in\mathcal{A}\left(  \mathcal{M}\right)  $ be a
pair of self-adjoint operators and $J$ be an operator in $\mathcal{M}$ with
$J\mathcal{D}\left(  H\right)  \subseteq\mathcal{D}\left(  H_{1}\right)  $.
Suppose there are $H$-bounded operator $G$ and $H_{1}$-bounded operator
$G_{1}$ in $\mathcal{A}\left(  \mathcal{M}\right)  $ satisfying $H_{1}%
J-JH=G_{1}^{\ast}G$. If
\[
\lim_{\varepsilon\rightarrow0}GR_{H}\left(  \lambda\pm i\varepsilon\right)
P_{ac}^{\infty}\left(  H\right)  f
\]
exist a.e. $\lambda\in\mathbb{R}$ for every $f\in\mathcal{H}$ and
\[
\lim_{\varepsilon\rightarrow0}\left\langle G_{1}\delta_{H_{1}}\left(
\lambda,\varepsilon\right)  P_{ac}^{\infty}\left(  H_{1}\right)
f_{1},g\right\rangle
\]
exist a.e. $\lambda\in\mathbb{R}$ for any $f_{1}$ and $g\in\mathcal{H}$, then
$\mathcal{U}_{\pm}\left(  H_{1},H;J\right)  $ exists.
\end{lemma}

\begin{proof}
Since $H_{1}J-JH=G_{1}^{\ast}G,$ by (\ref{g26})
\begin{align}
&  \frac{\varepsilon}{\pi}\left\langle JR_{H}\left(  \lambda\pm i\varepsilon
\right)  P_{ac}^{\infty}\left(  H\right)  f,R_{H_{1}}\left(  \lambda\pm
i\varepsilon\right)  P_{ac}^{\infty}\left(  H_{1}\right)  f_{1}\right\rangle
\nonumber\\
&  =\left\langle \left(  J+\left(  H_{1}J-JH\right)  R_{H}\left(  \lambda\pm
i\varepsilon\right)  \right)  P_{ac}^{\infty}\left(  H\right)  f,\delta
_{H_{1}}\left(  \lambda,\varepsilon\right)  P_{ac}^{\infty}\left(
H_{1}\right)  f_{1}\right\rangle \nonumber\\
&  =\left\langle \left(  J+G_{1}^{\ast}GR_{H}\left(  \lambda\pm i\varepsilon
\right)  \right)  P_{ac}^{\infty}\left(  H\right)  f,\delta_{H_{1}}\left(
\lambda,\varepsilon\right)  P_{ac}^{\infty}\left(  H_{1}\right)
f_{1}\right\rangle \nonumber\\
&  =\left\langle JP_{ac}^{\infty}\left(  H\right)  f,\delta_{H_{1}}\left(
\lambda,\varepsilon\right)  P_{ac}^{\infty}\left(  H_{1}\right)
f_{1}\right\rangle +\left\langle \left(  GR_{H}\left(  \lambda\pm
i\varepsilon\right)  \right)  P_{ac}^{\infty}\left(  H\right)  f,G_{1}%
\delta_{H_{1}}\left(  \lambda,\varepsilon\right)  P_{ac}^{\infty}\left(
H_{1}\right)  f_{1}\right\rangle . \label{g43}%
\end{align}
Then from (\ref{g6}),
\[
\lim_{\varepsilon\rightarrow0}\left\langle JP_{ac}^{\infty}\left(  H\right)
f,\delta_{H_{1}}\left(  \lambda,\varepsilon\right)  P_{ac}^{\infty}\left(
H_{1}\right)  f_{1}\right\rangle \text{ exists a.e}.\lambda\in\mathbb{R}%
\text{.}%
\]
Since
\[
\lim_{\varepsilon\rightarrow0}GR_{H}\left(  \lambda\pm i\varepsilon\right)
P_{ac}^{\infty}\left(  H\right)  f
\]
and
\[
\lim_{\varepsilon\rightarrow0}\left\langle G_{1}\delta_{H_{1}}\left(
\lambda,\varepsilon\right)  P_{ac}^{\infty}\left(  H_{1}\right)
f_{1},g\right\rangle
\]
exist a.e. $\lambda\in\mathbb{R}$ for every $f,f_{1}$ and $g\in\mathcal{H},$
we can easily check that
\[
\lim_{\varepsilon\rightarrow0}\left\langle \left(  GR_{H}\left(  \lambda\pm
i\varepsilon\right)  \right)  P_{ac}^{\infty}\left(  H\right)  f,G_{1}%
\delta_{H_{1}}\left(  \lambda,\varepsilon\right)  P_{ac}^{\infty}\left(
H_{1}\right)  f_{1}\right\rangle
\]
exists a.e. $\lambda\in\mathbb{R}$. Hence $\mathcal{U}_{\pm}\left(
H_{1},H;J\right)  $ is well-defined.
\end{proof}

\begin{lemma}
\label{M12}Let $H$, $H_{1}\in\mathcal{A}\left(  \mathcal{M}\right)  $ be a
pair of self-adjoint operators and $J$ be an operator in $\mathcal{M}$ with
$J\mathcal{D}\left(  H\right)  \subseteq\mathcal{D}\left(  H_{1}\right)  $.
Suppose there are $H$-bounded operator $G$ and $H_{1}$-bounded operator in
$\mathcal{A}\left(  \mathcal{M}\right)  $ satisfying $H_{1}J-JH=G_{1}^{\ast}%
G$. If
\begin{equation}
\lim_{\varepsilon\rightarrow0}GR_{H}\left(  \lambda\pm i\varepsilon\right)
P_{ac}^{\infty}\left(  H\right)  f \label{g44}%
\end{equation}
exist a.e. $\lambda\in\mathbb{R}$ for every $f\in\mathcal{H}$ and
\begin{equation}
\lim_{\varepsilon\rightarrow0}\left\langle G_{1}\delta_{H_{1}}\left(
\lambda,\varepsilon\right)  P_{ac}^{\infty}\left(  H_{1}\right)
f_{1},g\right\rangle \label{g45}%
\end{equation}
exist a.e. $\lambda\in\mathbb{R}$ for any $f_{1}$ and $g\in\mathcal{H},$ then
$\mathcal{U}_{\pm}\left(  H,H;J^{\ast}J\right)  $ exists and%
\[
\mathcal{U}_{\pm}^{\ast}\left(  H_{1},H;J\right)  \mathcal{U}_{\pm}\left(
H_{1},H;J\right)  =\mathcal{U}_{\pm}\left(  H,H;J^{\ast}J\right)  .
\]

\end{lemma}

\begin{proof}
By Lemma \ref{M11}, $\mathcal{U}_{\pm}\left(  H_{1},H;J\right)  $ exists. For
any Borel set $\Lambda$
\begin{equation}
\left\langle E_{H_{1}}\left(  \Lambda\right)  \cdot\mathcal{U}_{\pm}\left(
H_{1},H;J\right)  f,f_{1}\right\rangle =\int_{\Lambda}\lim_{\varepsilon
\rightarrow0}\left\langle \delta_{H_{1}}\left(  \lambda,\varepsilon\right)
\mathcal{U}_{\pm}\left(  H_{1},H;J\right)  f,f_{1}\right\rangle d\lambda
\label{G50}%
\end{equation}
where
\[
\lim_{\varepsilon\rightarrow0}\left\langle \delta_{H_{1}}\left(
\lambda,\varepsilon\right)  \mathcal{U}_{\pm}\left(  H_{1},H;J\right)
f,f_{1}\right\rangle \text{ exists}%
\]
by (\ref{g6}) and
\[
P_{ac}^{\infty}\left(  H_{1}\right)  \mathcal{U}_{\pm}\left(  H_{1}%
,H;J\right)  =\mathcal{U}_{\pm}\left(  H_{1},H;J\right)  .
\]
By Corollary \ref{R9}, we also have
\begin{align}
&  \left\langle E_{H_{1}}\left(  \Lambda\right)  \cdot\mathcal{U}_{\pm}\left(
H_{1},H;J\right)  f,f_{1}\right\rangle \nonumber\\
&  =\int_{\Lambda}\lim_{\varepsilon\rightarrow0}\frac{\varepsilon}{\pi
}\left\langle JR_{H}\left(  \lambda\pm i\varepsilon\right)  P_{ac}^{\infty
}\left(  H\right)  f,R_{H_{1}}\left(  \lambda\pm i\varepsilon\right)
P_{ac}^{\infty}\left(  H_{1}\right)  f_{1}\right\rangle d\lambda. \label{G52}%
\end{align}
Comparing two integrands of (\ref{G50}) and (\ref{G52}), we have
\begin{align}
&  \lim_{\varepsilon\rightarrow0}\left\langle \delta_{H_{1}}\left(
\lambda,\varepsilon\right)  \mathcal{U}_{\pm}\left(  H_{1},H;J\right)
f,f_{1}\right\rangle \nonumber\\
&  =\lim_{\varepsilon\rightarrow0}\frac{\varepsilon}{\pi}\left\langle
JR_{H}\left(  \lambda\pm i\varepsilon\right)  P_{ac}^{\infty}\left(  H\right)
f,R_{H_{1}}\left(  \lambda\pm i\varepsilon\right)  P_{ac}^{\infty}\left(
H_{1}\right)  f_{1}\right\rangle \text{ a.e. }\lambda\in\mathbb{R}.
\label{g27}%
\end{align}

So this equality holds only when
\[
\lim_{\varepsilon\rightarrow0}\frac{\varepsilon}{\pi}\left\langle
JR_{H}\left(  \lambda\pm i\varepsilon\right)  P_{ac}^{\infty}\left(  H\right)
f,R_{H_{1}}\left(  \lambda\pm i\varepsilon\right)  P_{ac}^{\infty}\left(
H_{1}\right)  f_{1}\right\rangle
\]
exists a.e. $\lambda\in\mathbb{R}$. By (\ref{g26}) and the fact that
$H_{1}J-JH=G_{1}^{\ast}G$, we have%
\begin{align*}
&  \frac{\varepsilon}{\pi}\left\langle JR_{H}\left(  \lambda\pm i\varepsilon
\right)  P_{ac}^{\infty}\left(  H\right)  f,R_{H_{1}}\left(  \lambda\pm
i\varepsilon\right)  P_{ac}^{\infty}\left(  H_{1}\right)  f_{1}\right\rangle
\\
&  =\left\langle JP_{ac}^{\infty}\left(  H\right)  f,\delta_{H_{1}}\left(
\lambda,\varepsilon\right)  P_{ac}^{\infty}\left(  H_{1}\right)
f_{1}\right\rangle +\left\langle \left(  GR_{H}\left(  \lambda\pm
i\varepsilon\right)  \right)  P_{ac}^{\infty}\left(  H\right)  f,G_{1}%
\delta_{H_{1}}\left(  \lambda,\varepsilon\right)  P_{ac}^{\infty}\left(
H_{1}\right)  f_{1}\right\rangle .
\end{align*}
Then by (\ref{g44}), (\ref{g45})and (\ref{g6}), we can conclude that
\begin{equation}
\lim_{\varepsilon\rightarrow0}\frac{\varepsilon}{\pi}\left\langle
JR_{H}\left(  \lambda\pm i\varepsilon\right)  P_{ac}^{\infty}\left(  H\right)
f,R_{H_{1}}\left(  \lambda\pm i\varepsilon\right)  P_{ac}^{\infty}\left(
H_{1}\right)  f_{1}\right\rangle \label{g68}%
\end{equation}
exists a.e. $\lambda\in\mathbb{R}$.

In (\ref{g68}), replace $f_{1}$ by $\mathcal{U}_{\pm}\left(  H_{1},H;J\right)
g,$ so from (\ref{g71}), (\ref{g26}) and (\ref{g27}) we have
\begin{align*}
&  \lim_{\varepsilon\rightarrow0}\frac{\varepsilon}{\pi}\left\langle
JR_{H}\left(  \lambda\pm i\varepsilon\right)  P_{ac}^{\infty}\left(  H\right)
f,R_{H_{1}}\left(  \lambda\pm i\varepsilon\right)  \mathcal{U}_{\pm}\left(
H_{1},H;J\right)  P_{ac}^{\infty}\left(  H\right)  g\right\rangle \\
&  =\lim_{\varepsilon\rightarrow0}\left\langle \left(  J+\left(
H_{1}J-JH\right)  R_{H}\left(  \lambda\pm i\varepsilon\right)  \right)
P_{ac}^{\infty}\left(  H\right)  f,\delta_{H_{1}}\left(  \lambda
,\varepsilon\right)  \mathcal{U}_{\pm}\left(  H_{1},H;J\right)  P_{ac}%
^{\infty}\left(  H\right)  g\right\rangle \\
&  =\lim_{\varepsilon\rightarrow0}\frac{\varepsilon}{\pi}\left\langle \left(
J+\left(  H_{1}J-JH\right)  R_{H}\left(  \lambda\pm i\varepsilon\right)
\right)  P_{ac}^{\infty}\left(  H\right)  f,R_{H_{1}}\left(  \lambda\mp
i\varepsilon\right)  JR_{H}\left(  \lambda\pm i\varepsilon\right)
P_{ac}^{\infty}\left(  H\right)  g\right\rangle \\
&  =\lim_{\varepsilon\rightarrow0}\left\langle \left(  J+\left(
H_{1}J-JH\right)  R_{H}\left(  \lambda\pm i\varepsilon\right)  \right)
P_{ac}^{\infty}\left(  H\right)  f,\delta_{H_{1}}\left(  \lambda\right)
R_{H_{1}}^{-1}\left(  \lambda\pm i\varepsilon\right)  JR_{H}\left(  \lambda\pm
i\varepsilon\right)  P_{ac}^{\infty}\left(  H\right)  g\right\rangle \\
&  =\lim_{\varepsilon\rightarrow0}\left\langle JR_{H}\left(  \lambda\pm
i\varepsilon\right)  P_{ac}^{\infty}\left(  H\right)  f,JR_{H}\left(
\lambda\pm i\varepsilon\right)  P_{ac}^{\infty}\left(  H\right)
g\right\rangle \\
&  =\lim_{\varepsilon\rightarrow0}\frac{\varepsilon}{\pi}\left\langle
R_{H}\left(  \lambda\mp i\varepsilon\right)  J^{\ast}JR_{H}\left(  \lambda\pm
i\varepsilon\right)  P_{ac}^{\infty}\left(  H\right)  f,P_{ac}^{\infty}\left(
H\right)  g\right\rangle \text{ a.e. }\lambda\in\mathbb{R}.
\end{align*}
Hence applying the definition of $\mathcal{U}_{\pm}\left(  H_{1},H;J\right)
,$ we have
\begin{align*}
&  \left\langle \mathcal{U}_{\pm}\left(  H_{1},H;J\right)  f,\mathcal{U}_{\pm
}\left(  H_{1},H;J\right)  g\right\rangle \\
&  =\int_{-\infty}^{\infty}\lim_{\varepsilon\rightarrow0}\frac{\varepsilon
}{\pi}\left\langle JR_{H}\left(  \lambda\pm i\varepsilon\right)
P_{ac}^{\infty}\left(  H\right)  f,R_{H_{1}}\left(  \lambda\pm i\varepsilon
\right)  \mathcal{U}_{\pm}\left(  H_{1},H;J\right)  P_{ac}^{\infty}\left(
H\right)  g\right\rangle \\
&  =\int_{-\infty}^{\infty}\lim_{\varepsilon\rightarrow0}\frac{\varepsilon
}{\pi}\left\langle R_{H}\left(  \lambda\mp i\varepsilon\right)  J^{\ast}%
JR_{H}\left(  \lambda\pm i\varepsilon\right)  P_{ac}^{\infty}\left(  H\right)
f,P_{ac}^{\infty}\left(  H\right)  g\right\rangle \\
&  =\left\langle \mathcal{U}_{\pm}\left(  H,H;J^{\ast}J\right)
f,g\right\rangle .
\end{align*}
It implies that\ \
\[
\mathcal{U}_{\pm}^{\ast}\left(  H_{1},H;J\right)  \mathcal{U}_{\pm}\left(
H_{1},H;J\right)  =\mathcal{U}_{\pm}\left(  H,H;J^{\ast}J\right)  .
\]

\end{proof}

According to Theorem \ref{R10}, for showing the existence of $W_{\pm}%
(H_{1},H;J)$ by a stationary method, we also need to prove the existence of
$\widetilde{W}_{\pm}\left(  H_{1},H;J\right)  $ and $\widetilde{W}_{\pm
}\left(  H,H;J^{\ast}J\right)  $ without using time variable $t$ explicitly.

\subsection{The Kato-Rosenblum theorem in $\mathcal{M}$ by a stationary
approach}

The results below involve noncommutative $\mathcal{L}^{p}$-spaces associated
to a semifinite von Neumann algebra $\mathcal{M},$ so we refer the reader to
\cite{Pisier} for more details about it.

\begin{remark}
\label{M13}For a separable Hilbert space $\mathcal{H},$ we denote by $H_{\pm
}^{2}\left(  \mathcal{H}\right)  $ the class of functions with values in
$\mathcal{H},$ holomorphic on upper (lower) half-plane and such that
\[
\sup_{\varepsilon>0}\int_{\mathbb{R}}\left\Vert u\left(  \lambda\pm
i\varepsilon\right)  \right\Vert ^{2}d\lambda<+\infty.
\]
Then by the result in Section 1 of Chapter V in \cite{SF}, we know that the
radial limit exists almost everywhere, i.e., $\lim_{\varepsilon\rightarrow
0}u\left(  \lambda\pm i\varepsilon\right)  $ exists a.e. $\lambda\in
\mathbb{R}$.
\end{remark}

\begin{lemma}
\label{M14}Let $H\in\mathcal{A}\left(  \mathcal{M}\right)  $ be a self-adjoint
operator and $A\in\mathcal{L}^{2}\left(  \mathcal{M},\tau\right)
\cap\mathcal{M}.$ Then%
\[
s.o.t\text{-}\lim_{\varepsilon\rightarrow0}AR_{H}\left(  \lambda\pm
i\varepsilon\right)  P_{ac}^{\infty}\left(  H\right)  \text{ }%
\]
and
\[
s.o.t\text{-}\lim_{\varepsilon\rightarrow0}A\delta_{H}\left(  \lambda
,\varepsilon\right)  P_{ac}^{\infty}\left(  H\right)
\]
exist in the strong operator topology a.e. $\lambda\in\mathbb{R}$.
\end{lemma}

\begin{proof}
By Remark \ref{M26}, and Theorem \ref{K1}, we get%
\[
\sup_{\left\Vert f\right\Vert =1}\frac{1}{2\pi}\int_{\mathbb{R}}\left\Vert
P\omega_{n}\left(  H\right)  e^{-itH}f\right\Vert ^{2}dt=\sup_{\Lambda
\subseteq\mathbb{R}}\frac{\left\Vert P\omega_{n}\left(  H\right)  E_{H}\left(
\Lambda\right)  \omega_{n}\left(  H\right)  P\right\Vert }{\left\vert
\Lambda\right\vert }\leq\frac{n}{\left(  2\pi\right)  ^{2}}.
\]
Hence by (\ref{g7}), we have%
\begin{align*}
&  \sup_{\left\Vert f\right\Vert =1}\frac{1}{2\pi}\int_{\mathbb{R}}\left\Vert
P\omega_{n}\left(  H\right)  e^{-itH}f\right\Vert ^{2}dt\\
&  =\frac{1}{\left(  2\pi\right)  ^{2}}\sup_{\left\Vert f\right\Vert
=1,\varepsilon>0}\int_{\mathbb{R}}\left(  \left\Vert P\omega_{n}\left(
H\right)  R_{H}(\lambda\pm i\varepsilon)f\right\Vert ^{2}\right)  d\lambda
\leq\frac{n}{\left(  2\pi\right)  ^{2}}.
\end{align*}
From Lemma 2.1.1 in \cite{Li}, there is a sequence $\left\{  x_{m}\right\}
_{m\in\mathbb{N}}$ of $\mathcal{H}$ such that
\[
\left\Vert A\right\Vert _{2}^{2}=\tau\left(  A^{\ast}A\right)  =\sum
\left\langle A^{\ast}Ax_{m},x_{m}\right\rangle
\]
and
\[
\vee\left\{  A^{\prime}x_{m}:A^{\prime}\in\mathcal{M}^{\prime}\text{ and }%
m\in\mathbb{N}\right\}
\]
where $\mathcal{M}^{\prime}$ is the commutant of $\mathcal{M}.$ Then for these
$\left\{  x_{m}\right\}  _{m\in\mathbb{N}}$ , we have
\[
\int_{\mathbb{R}}\left(  \left\Vert P\omega_{n}\left(  H\right)  R_{H}%
(\lambda\pm i\varepsilon)Ax_{m}\right\Vert ^{2}\right)  d\lambda\leq\frac
{n}{\left(  2\pi\right)  ^{2}}\left\Vert Ax_{m}\right\Vert ^{2}\leq\frac
{n}{\left(  2\pi\right)  ^{2}}\left\Vert A\right\Vert _{2}^{2}%
\]
for $P\in\mathcal{P}_{a.c}^{\infty}\left(  H\right)  .$ We further note that
for every $A\in\mathcal{L}^{2}\left(  \mathcal{M},\tau\right)  \cap
\mathcal{M},$
\begin{align*}
\int_{\mathbb{R}}\left\Vert P\omega_{n}\left(  H\right)  R_{H}(\lambda\pm
i\varepsilon)A\right\Vert _{2}^{2}d\lambda &  =\int_{\mathbb{R}}\sum
_{m}\left\Vert P\omega_{n}\left(  H\right)  R_{H}(\lambda\pm i\varepsilon
)Ax_{m}\right\Vert _{2}^{2}d\lambda\\
&  =\sum_{m}\int_{\mathbb{R}}\left\Vert P\omega_{n}\left(  H\right)
R_{H}(\lambda\pm i\varepsilon)Ax_{m}\right\Vert _{2}^{2}d\lambda\\
&  \leq\frac{n}{2\pi}\sum\left\Vert Ax_{m}\right\Vert ^{2}\leq\frac{n}{2\pi
}\left\Vert A\right\Vert _{2}^{2}.
\end{align*}
Combing it with the equality
\[
\left\Vert X\right\Vert _{2}^{2}=\tau\left(  X^{\ast}X\right)  =\tau\left(
XX^{\ast}\right)  =\left\Vert X^{\ast}\right\Vert _{2}^{2}\text{ for very
}X\in\mathcal{M},
\]
we get the following inequality
\begin{align*}
\int_{\mathbb{R}}\left\Vert AR_{H}(\lambda\pm i\varepsilon)\omega_{n}\left(
H\right)  Px_{m}\right\Vert ^{2}d\lambda &  \leq\int_{\mathbb{R}}\left\Vert
AR_{H}(\lambda\pm i\varepsilon)\omega_{n}\left(  H\right)  P\right\Vert
_{2}^{2}d\lambda\\
&  =\int_{\mathbb{R}}\left\Vert P\omega_{n}\left(  H\right)  R_{H}(\lambda\mp
i\varepsilon)A\right\Vert _{2}^{2}d\lambda\\
&  \leq\frac{n}{2\pi}\left\Vert A\right\Vert _{2}^{2}.
\end{align*}
It implies that the vector-valued function $AR_{H}(\lambda\pm i\varepsilon
)\omega_{n}\left(  H\right)  Px_{m}$ belongs to the Hardy classes $H_{\pm}%
^{2}\left(  \mathcal{H}\right)  $ in the upper and lower half planes. By
Remark \ref{M13}, the radial limit values of functions in $H_{\pm}^{2}\left(
\mathcal{H}\right)  $ exist a.e. $\lambda\in\mathbb{R},$ therefore
\[
\lim_{\varepsilon\rightarrow0}AR_{H}\left(  \lambda\pm i\varepsilon\right)
\omega_{n}\left(  H\right)  Px_{m}\text{ exists a.e. }\lambda\in
\mathbb{R}\text{ for every }x_{m}.
\]
Since the linear span of the set $\left\{  A^{\prime}x_{m}:A^{\prime}%
\in\mathcal{M}^{\prime}\text{ and }m\in\mathbb{N}\right\}  $ is dense in
$\mathcal{H},$ we have
\[
\lim_{\varepsilon\rightarrow0}AR_{H}\left(  \lambda\pm i\varepsilon\right)
\omega_{n}\left(  H\right)  PA^{\prime}x_{m}=A^{\prime}\lim_{\varepsilon
\rightarrow0}AR_{H}\left(  \lambda\pm i\varepsilon\right)  \omega_{n}\left(
H\right)  Px_{m}\text{ exists}%
\]
and then this indicates that
\[
s.o.t\text{-}\lim_{\varepsilon\rightarrow0}AR_{H}\left(  \lambda\pm
i\varepsilon\right)  \omega_{n}\left(  H\right)  P\text{ exists }%
\]
in strong operator topology. From the fact that $\omega_{n}\left(  H\right)
\rightarrow I$ $(n\rightarrow\infty)$ in Lemma \ref{M5}, we can conclude that
\[
s.o.t\text{-}\lim_{\varepsilon\rightarrow0}AR_{H}\left(  \lambda\pm
i\varepsilon\right)  P\text{ exists for }A\in\mathcal{L}^{2}\left(
\mathcal{M},\tau\right)  \cap\mathcal{M}\text{ and }P\in\mathcal{P}%
_{ac}^{\infty}\left(  H\right)  .
\]
Since $P_{ac}^{\infty}\left(  H\right)  =\vee\left\{  P:P\in\mathcal{P}%
_{ac}^{\infty}\left(  H\right)  \right\}  ,$ we get
\[
s.o.t\text{-}\lim_{\varepsilon\rightarrow0}AR_{H}\left(  \lambda\pm
i\varepsilon\right)  P_{ac}^{\infty}\left(  H\right)  \text{ exists for }%
A\in\mathcal{L}^{2}\left(  \mathcal{M},\tau\right)  \cap\mathcal{M}.
\]
Note that $\delta_{H}\left(  \lambda,\varepsilon\right)  =\frac{1}{2\pi
i}\left[  R_{H}\left(  \lambda+i\varepsilon\right)  -R_{H}\left(
\lambda-i\varepsilon\right)  \right]  ,$ so we can conclude that
\[
\lim_{\varepsilon\rightarrow0}A\delta_{H}\left(  \lambda,\varepsilon\right)
P_{ac}^{\infty}\left(  H\right)  x=\frac{1}{2\pi i}(\lim_{\varepsilon
\rightarrow0}AR_{H}\left(  \lambda+i\varepsilon\right)  P_{ac}^{\infty}\left(
H\right)  x-\lim_{\varepsilon\rightarrow0}AR_{H}\left(  \lambda-i\varepsilon
\right)  P_{ac}^{\infty}\left(  H\right)  x)\text{ }%
\]
exists for $A\in\mathcal{L}^{2}\left(  \mathcal{M},\tau\right)  \cap
\mathcal{M}.$ The proof is completed.
\end{proof}

\begin{remark}
By Lemma 2.1.6 in \cite{Li2}, we knows that $\mathcal{L}^{p}\left(
\mathcal{M},\tau\right)  \cap\mathcal{M}$ is a two-sided ideal of
$\mathcal{M}$ for $1\leq p<\infty.$
\end{remark}

\begin{remark}
\label{R30}If $G\in\mathcal{M},$ then by (\ref{g6})
\[
\lim_{\varepsilon\rightarrow0}\left\langle G_{1}\delta_{H_{1}}\left(
\lambda,\varepsilon\right)  P_{ac}^{\infty}\left(  H_{1}\right)
f_{1},g\right\rangle =\lim_{\varepsilon\rightarrow0}\left\langle \delta
_{H_{1}}\left(  \lambda,\varepsilon\right)  P_{ac}^{\infty}\left(
H_{1}\right)  f_{1},G_{1}g\right\rangle
\]
exists a.e. $\lambda\in\mathbb{R}$ for any $f_{1}$ and $g\in\mathcal{H}.$
\end{remark}

\begin{theorem}
\label{M15}Let $H$, $H_{1}\in\mathcal{A}\left(  \mathcal{M}\right)  $ be a
pair of self-adjoint operators and $J$ be an operator in $\mathcal{M}$ with
$J\mathcal{D}\left(  H\right)  \subseteq\mathcal{D}\left(  H_{1}\right)  $.
Assume $H_{1}J-JH\in\mathcal{L}^{1}\left(  \mathcal{M},\tau\right)
\cap\mathcal{M}$, then $\mathcal{U}_{\pm}\left(  H_{1},H;J\right)
\mathcal{\ }$and $\mathcal{U}_{\pm}\left(  H,H;J^{\ast}J\right)  $ both exist
and%
\[
\mathcal{U}_{\pm}^{\ast}\left(  H_{1},H;J\right)  \mathcal{U}_{\pm}\left(
H_{1},H;J\right)  =\mathcal{U}_{\pm}\left(  H,H;J^{\ast}J\right)  .
\]

\end{theorem}

\begin{proof}
Let $G=\left\vert H_{1}J-JH\right\vert ^{1/2}\in\mathcal{L}^{2}\left(
\mathcal{M},\tau\right)  \cap\mathcal{M}$ and $G_{1}^{\ast}=V\left\vert
H_{1}J-JH\right\vert ^{1/2}\in\mathcal{L}^{2}\left(  \mathcal{M},\tau\right)
\cap\mathcal{M}$ for some partial isometry $V$ in $\mathcal{M}$. Then the
proof is completed by Remark \ref{R30}, Lemma \ref{M14}, Lemma \ref{M11} and
Lemma \ref{M12}.
\end{proof}

According Theorem \ref{R10}, for showing the existence of $W_{\pm}(H_{1},H;J)$
by the stationary approach, we first need to show the existence of
$\widetilde{W}_{\pm}\left(  H_{1},H;J\right)  $ without depending on time
variable $t$ explicitly$.$ For doing this, we need several lemmas.

\begin{lemma}
\label{M16}(Lemma 2.5.1 in \cite{Li})Let $H$, $H_{1}\in\mathcal{A}\left(
\mathcal{M}\right)  $ be a pair of self-adjoint operators, $J$ be an operator
in $\mathcal{M}$ with $J\mathcal{D}\left(  H\right)  \subseteq\mathcal{D}%
\left(  H_{1}\right)  $ and $H_{1}J-JH\in\mathcal{M}.$ Let $W_{J}\left(
t\right)  =e^{itH_{1}}Je^{-itH},$ for $t\in\mathbb{R}.$ Then, for all
$f\in\mathcal{H}$ and $s,w\in\mathbb{R},$ the mapping $t\longmapsto
e^{itH_{1}}\left(  H_{1}J-JH\right)  e^{-itH}f$ from $\left[  s,w\right]  $
into $\mathcal{H}$ is Bochner integrable with%
\[
\left(  W_{J}\left(  w\right)  -W_{J}\left(  s\right)  \right)  f=i\int
_{s}^{w}e^{itH_{1}}\left(  H_{1}J-JH\right)  e^{-itH}fdt.
\]

\end{lemma}

\begin{lemma}
\label{M17} Let $H\in\mathcal{A}\left(  \mathcal{M}\right)  $ be a
self-adjoint operator and $G$ be an operator in $\mathcal{M}$. Then there is a
linear manifold $\mathcal{D}$ in $\mathcal{H}_{ac}^{\infty}\left(  H\right)  $
with $\overline{\mathcal{D}}=\mathcal{H}_{ac}^{\infty}\left(  H\right)  $ such
that
\[
\int_{-\infty}^{\infty}\left\Vert Ge^{-itH}g\right\Vert ^{2}dt<\infty,\text{
}g\in\mathcal{D}.
\]

\end{lemma}

\begin{proof}
For any $f\in\mathcal{H}_{ac}^{\infty}\left(  H\right)  ,$ by (\ref{g6}), we
have
\[
\lim_{\varepsilon\rightarrow0}\left\langle G\delta(\lambda,\varepsilon
)f,h\right\rangle =\frac{d\left\langle GE_{H}\left(  \lambda\right)
f,g\right\rangle }{d\lambda}\text{ a.e. }\lambda\in\mathbb{R}\text{ for any
}h\in\mathcal{H}.
\]
Let $F_{f}\left(  \lambda\right)  \in\mathcal{H}$ be the weak limit of
$G\delta(\lambda,\varepsilon)f,$ i.e. $\lim_{\varepsilon\rightarrow
0}\left\langle G\delta(\lambda,\varepsilon)f,h\right\rangle =\left\langle
F_{f}\left(  \lambda\right)  ,h\right\rangle $ a.e. $\lambda\in\mathbb{R}$ for
every $h\in\mathcal{H}.$ We set
\[
X_{N,n}(f)=\left\{  \lambda:\left\vert \lambda\right\vert \leq n,\left\Vert
F_{f}\left(  \lambda\right)  \right\Vert \leq N\right\}
\]
and $\mathcal{D}$ to be the set of linear combination of all elements of the
form $g=E\left(  X_{N,n}\right)  f$ for $f\in\mathcal{H}_{ac}^{\infty}\left(
H\right)  $ and $n,N\in\mathbb{N}.$ Since for $f\in\mathcal{H}_{ac}^{\infty
}\left(  H\right)  $ and $n,N\in\mathbb{N},$
\[
E\left(  X_{N,n}\right)  f=E\left(  X_{N,n}\right)  P_{ac}^{\infty}\left(
H\right)  f=P_{ac}^{\infty}\left(  H\right)  E\left(  X_{N,n}\right)  f
\]
by Remark \ref{I1}, we have $\mathcal{D\subset H}_{ac}^{\infty}\left(
H\right)  .$ Note
\[
\lim_{N\rightarrow\infty}\left\vert \left(  -n,n\right)  \backslash
X_{N,n}\right\vert =0,
\]
then $f$ can be approximated by the elements $E\left(  X_{N,n}\right)  f$ for
$f\in\mathcal{H}_{ac}^{\infty}\left(  H\right)  .$ Hence $\overline
{\mathcal{D}}=\mathcal{H}_{ac}^{\infty}\left(  H\right)  .$

Let $\left\{  e_{i}\right\}  _{i\in\mathbb{Z}}$ be an orthonormal basis in
$\mathcal{H}.$ By \ref{g1}, for $g=E\left(  X_{N,n}\right)  f$
\begin{align*}
\left\langle Ge^{-itH}g,e_{i}\right\rangle  &  =\int_{\mathbb{R}}e^{-i\lambda
t}d\left\langle GE_{H}\left(  \lambda\right)  g,e_{i}\right\rangle \\
&  =\int_{X_{N,n}\left(  g\right)  }e^{-i\lambda t}\frac{d\left\langle
GE_{H}\left(  \lambda\right)  g,e_{i}\right\rangle }{d\lambda}d\lambda\\
&  =\int_{X_{N,n}\left(  g\right)  }e^{-i\lambda t}\;\left\langle F_{g}\left(
\lambda\right)  ,e_{i}\right\rangle d\lambda.
\end{align*}
Then by the Parseval equality, for each $i\in\mathbb{Z}$
\[
\int_{\mathbb{R}}\left\vert \left\langle Ge^{-itH}g,e_{i}\right\rangle
\right\vert ^{2}dt=2\pi\int_{X_{N,n}\left(  g\right)  }\left\vert \left\langle
F_{g}\left(  \lambda\right)  ,e_{i}\right\rangle \right\vert ^{2}d\lambda
\]
Hence
\[
\int_{\mathbb{R}}\left\Vert Ge^{-itH}g\right\Vert ^{2}dt=2\pi\int
_{X_{N,n}\left(  g\right)  }\left\Vert F_{g}\left(  \lambda\right)
\right\Vert ^{2}d\lambda\leq4\pi N^{2}n.
\]
Therefore we have
\[
\int_{-\infty}^{\infty}\left\Vert Ge^{-itH}g\right\Vert ^{2}dt<\infty,\text{
}g\in\mathcal{D}.
\]

\end{proof}

\begin{theorem}
\label{M18}Let the operators $H,H_{1}\in\mathcal{A}\left(  \mathcal{M}\right)
$ be a pair of self-adjoint operators and $J$ be an operator in $\mathcal{M}$
with $J\mathcal{D}\left(  H\right)  \subseteq\mathcal{D}\left(  H_{1}\right)
.$ Let $W_{J}\left(  t\right)  =e^{itH_{1}}Je^{-itH},$ for $t\in\mathbb{R}.$
If $H_{1}J-JH=G_{1}^{\ast}G$ for $G_{1}$ and $G$ in $\mathcal{M}$, then the
generalized weak wave operator $\widetilde{W}_{\pm}\left(  H_{1},H;J\right)  $ exists.
\end{theorem}

\begin{proof}
By Lemma \ref{M17}, there are linear space $\mathcal{D\subseteq H}%
_{ac}^{\infty}\left(  H\right)  $ and $\mathcal{D}_{1}\mathcal{\subseteq
H}_{ac}^{\infty}\left(  H_{1}\right)  $ with $\overline{\mathcal{D}%
}=\mathcal{H}_{ac}^{\infty}\left(  H\right)  $ and $\overline{\mathcal{D}_{1}%
}=\mathcal{H}_{ac}^{\infty}\left(  H_{1}\right)  .$ Then for $f\in\mathcal{D}$
and $g\in\mathcal{D}_{1}$
\begin{align*}
\left\vert \left\langle W_{J}\left(  w\right)  -W_{J}\left(  s\right)
f,g\right\rangle \right\vert  &  =\left\vert \int_{s}^{w}\left\langle
e^{itH_{1}}\left(  H_{1}J-JH\right)  e^{-itH}f,g\right\rangle dt\right\vert \\
&  \leq\int_{s}^{w}\left\vert \left\langle Ge^{-itH}f,G_{1}e^{-itH_{1}%
}g\right\rangle \right\vert d\lambda\\
&  \leq\left(  \int_{s}^{w}\left\Vert Ge^{-itH}f\right\Vert ^{2}dt\cdot
\int_{s}^{t}\left\Vert G_{1}e^{-itH_{1}}g\right\Vert ^{2}dt\right)  ^{1/2}%
\end{align*}
and
\[
\int_{s}^{w}\left\Vert Ge^{-itH}f\right\Vert ^{2}dt\rightarrow0,\int_{s}%
^{w}\left\Vert G_{1}e^{-itH_{1}}g\right\Vert ^{2}dt\rightarrow0
\]
as $s,w\rightarrow\pm\infty.$ Hence
\[
\lim_{t\rightarrow\pm\infty}\left\langle W_{J}\left(  t\right)
f,g\right\rangle =\lim_{t\rightarrow\pm\infty}\left\langle W_{J}\left(
t\right)  P_{ac}^{\infty}\left(  H\right)  f,P_{ac}^{\infty}\left(
H_{1}\right)  g\right\rangle
\]
exists for $f\in\mathcal{D}$ and $g\in\mathcal{D}_{1}.$ Since
\[
\overline{\mathcal{D}}=\mathcal{H}_{ac}^{\infty}\left(  H\right)  \text{ and
}\overline{\mathcal{D}_{1}}=\mathcal{H}_{ac}^{\infty}\left(  H_{1}\right)  ,
\]
we have%
\[
\lim_{t\rightarrow\pm\infty}\left\langle P_{ac}^{\infty}\left(  H_{1}\right)
W_{J}\left(  t\right)  P_{ac}^{\infty}\left(  H\right)  f,g\right\rangle
\]
exist for any $f,g\in\mathcal{H}.$ Therefore $\widetilde{W}_{\pm}\left(
H_{1},H;J\right)  $ exists.
\end{proof}

\begin{corollary}
\label{M19}Let the operators $H,H_{1}\in\mathcal{A}\left(  \mathcal{M}\right)
$ be a pair of self-adjoint operators and $J$ be an operator in $\mathcal{M}$
with $J\mathcal{D}\left(  H\right)  \subseteq\mathcal{D}\left(  H_{1}\right)
.$ If $H_{1}J-JH$ $\in\mathcal{L}^{1}\left(  \mathcal{M},\tau\right)
\cap\mathcal{M},$ then the generalized weak wave operator $\widetilde{W}_{\pm
}=\widetilde{W}_{\pm}\left(  H,H;J^{\ast}J\right)  $ exists.
\end{corollary}

\begin{proof}
Since%
\[
HJ^{\ast}J-J^{\ast}JH=J^{\ast}\left(  H_{1}J-JH\right)  -\left(  J^{\ast}%
H_{1}-HJ^{\ast}\right)  J\in\mathcal{M},
\]
we have $\widetilde{W}_{\pm}=\widetilde{W}_{\pm}\left(  H,H;J^{\ast}J\right)
$ exists by Theorem \ref{M18}.
\end{proof}

Next result is the analogue of Kato-Rosenblum Theorem in a semifinite von
Neumann algebra $\mathcal{M}$ which is first proved in \cite{Li} by a
time-dependent approach. One of our main purpose of this article is to obtain
this result by a stationary approach. Now we are ready to show it here. 

\begin{theorem}
\label{T1}(Theorem 5.2.5 in \cite{Li}) Let $H$, $H_{1}\in\mathcal{A}\left(
\mathcal{M}\right)  $ be a pair of self-adjoint operators and $J$ be an
operator in $\mathcal{M}$ with $J\mathcal{D}\left(  H\right)  \subseteq
\mathcal{D}\left(  H_{1}\right)  $. Assume $H_{1}-H\in\mathcal{L}^{1}\left(
\mathcal{M},\tau\right)  \cap\mathcal{M}$, then
\[
W_{\pm}\overset{\triangle}{=}W_{\pm}\left(  H_{1},H\right)  \text{ exists in
}\mathcal{M}.
\]
Moreover, $W_{\pm}^{\ast}W_{\pm}=P_{ac}^{\infty}\left(  H\right)  $, $W_{\pm
}W_{_{\pm}}^{\ast}=P_{ac}^{\infty}\left(  H_{1}\right)  $ and $W_{\pm}HW_{\pm
}^{\ast}=H_{1}P_{ac}^{\infty}\left(  H_{1}\right)  .$
\end{theorem}

\begin{proof}
Let $J=I.$ Combing Theorem \ref{M18}, Theorem \ref{M15} and Theorem \ref{R10},
we know that $W_{\pm}\left(  H_{1},H\right)  $ and $W_{\pm}\left(
H,H_{1}\right)  $ both exist. By Theorem \ref{I4}, we also have
\[
W_{\pm}^{\ast}W_{\pm}=W_{\pm}^{\ast}\left(  H_{1},H\right)  W_{\pm}\left(
H_{1},H\right)  =P_{ac}^{\infty}\left(  H\right)
\]
and
\[
W_{\pm}^{\ast}\left(  H,H_{1}\right)  W_{\pm}\left(  H,H_{1}\right)
=P_{ac}^{\infty}\left(  H_{1}\right)  .
\]
Since $H_{1}-H$ $\in\mathcal{L}^{1}\left(  \mathcal{M},\tau\right)
\cap\mathcal{M},$ we have $\widetilde{W}_{\pm}\left(  H_{1},H\right)  $ and
$\widetilde{W}_{\pm}\left(  H,H_{1}\right)  $ both exist and
\begin{align*}
W_{\pm}^{\ast}  &  =W_{\pm}^{\ast}\left(  H_{1},H\right)  =\widetilde{W}_{\pm
}^{\ast}\left(  H_{1},H\right)  =\widetilde{W}_{\pm}\left(  H,H_{1}\right)
=W_{\pm}\left(  H,H_{1}\right)  ;\\
W_{\pm}^{\ast}\left(  H,H_{1}\right)   &  =\widetilde{W}_{\pm}^{\ast}\left(
H,H_{1}\right)  =\widetilde{W}_{\pm}\left(  H_{1},H\right)  =W_{\pm}\left(
H_{1},H\right)  =W_{\pm}%
\end{align*}
by equality (\ref{g33}). Thus
\[
W_{\pm}^{\ast}W_{\pm}=P_{ac}^{\infty}\left(  H\right)  ,\text{ }W_{\pm
}W_{_{\pm}}^{\ast}=P_{ac}^{\infty}\left(  H_{1}\right)  .
\]
Meanwhile, by Theorem \ref{I3},
\[
W_{\pm}HW_{\pm}^{\ast}=HW_{\pm}W_{\pm}^{\ast}=P_{ac}^{\infty}\left(
H_{1}\right)  .
\]
So the proof is completed.
\end{proof}

\begin{remark}
For a self-adjoint $H\in\mathcal{A}\left(  \mathcal{M}\right)  ,$ if there is
a $H$-smooth operator in $\mathcal{M},$ then $H$ is not a sum of a diagonal
operator in $\mathcal{M}$ and an operator in $\mathcal{M\cap L}^{1}\left(
\mathcal{M},\tau\right)  $ by Theorem \ref{T1} and Theorem \ref{M7}.
\end{remark}

\vspace{8cm}

\section{Appendix\vspace{1cm}}

\subsection{Relation between $\widetilde{W}_{\pm}$ and $W_{\pm}$}

\begin{lemma}
\label{A1}Let $H$, $H_{1}\in\mathcal{A}\left(  \mathcal{M}\right)  $ be a pair
of self-adjoint operators and $J$ be an operator in $\mathcal{M}$. If
\[
\widetilde{W}_{\pm}(H_{1},H;J),\text{ }\widetilde{W}_{\pm}(H,H;J^{\ast}J)
\]
exist and the equality
\[
\widetilde{W}_{\pm}^{\ast}(H_{1},H;J)\widetilde{W}_{\pm}(H_{1},H;J)=\widetilde
{W}_{\pm}(H,H;J^{\ast}J)
\]
holds, then $W_{\pm}(H_{1},H;J)$ exists and
\[
W_{\pm}(H_{1},H;J)=\widetilde{W}_{\pm}(H_{1},H;J).
\]

\end{lemma}

\begin{proof}
Suppose $\widetilde{W}_{\pm}\left(  H_{1},H;J\right)  $ exists. For
$f\in\mathcal{H},$
\begin{align}
&  \left\Vert e^{itH_{1}}Je^{-itH}P_{ac}^{\infty}\left(  H\right)
f-\widetilde{W}_{\pm}\left(  H_{1},H;J\right)  f\right\Vert ^{2}\nonumber\\
&  =\left\langle P_{ac}^{\infty}\left(  H\right)  e^{itH}J^{\ast}%
Je^{-itH}P_{ac}^{\infty}\left(  H\right)  f,f\right\rangle \nonumber\\
&  -2\operatorname{Re}\left\langle e^{itH_{1}}Je^{-itH}P_{ac}^{\infty}\left(
H\right)  f,\widetilde{W}_{\pm}\left(  H_{1},H;J\right)  f\right\rangle
+\left\Vert \widetilde{W}_{\pm}\left(  H_{1},H;J\right)  f\right\Vert ^{2}.
\label{g20}%
\end{align}

By the fact that $\widetilde{W}_{\pm}\left(  H_{1},H;J\right)  =P_{ac}%
^{\infty}\left(  H_{1}\right)  \widetilde{W}_{\pm}\left(  H_{1},H;J\right)  ,$
we conclude that
\[
\operatorname{Re}\left\langle e^{itH_{1}}Je^{-itH}P_{ac}^{\infty}\left(
H\right)  f,\widetilde{W}_{\pm}\left(  H_{1},H;J\right)  f\right\rangle
\rightarrow\left\Vert \widetilde{W}_{\pm}\left(  H_{1},H;J\right)
f\right\Vert ^{2}%
\]
as $t\rightarrow\pm\infty.$ Therefore by the assumption that $\widetilde
{W}_{\pm}\left(  H,H;J^{\ast}J\right)  $ exists and
\[
\widetilde{W}_{\pm}^{\ast}(H_{1},H;J)\widetilde{W}_{\pm}(H_{1},H;J)=\widetilde
{W}_{\pm}(H,H;J^{\ast}J),
\]
we conclude that the right side of the equality (\ref{g20}) is zero$,$
therefore
\[
\lim_{t\rightarrow\pm\infty}\left\Vert e^{itH_{1}}Je^{-itH}P_{ac}^{\infty
}\left(  H\right)  f-\widetilde{W}_{\pm}\left(  H_{1},H;J\right)  f\right\Vert
=0,
\]
it implies that
\[
W_{\pm}(H_{1},H;J)f=\lim_{t\rightarrow\pm\infty}e^{itH_{1}}Je^{-itH}%
P_{ac}^{\infty}\left(  H\right)  f=\widetilde{W}_{\pm}\left(  H_{1}%
,H;J\right)  f.
\]
The proof is completed.
\end{proof}

\subsection{\bigskip Relation among $\mathcal{U}_{\pm},$ $\widetilde
{\mathfrak{U}}_{\pm},$ $\widetilde{W}_{\pm}$ and $W_{\pm}$}

The idea of giving the definition of generalized stationary wave operators in
$\mathcal{M}$ is same as the idea given in Section 2.2 of \cite{Y1}. So we
summarize it here briefly.

Suppose that a nonnegative function $\omega\left(  t\right)  ,$ $t\geq0,$ is
normalized by the condition
\[
\int_{0}^{\infty}\omega\left(  t\right)  dt=1.
\]
We introduce the averaging kernel $\omega_{\varepsilon}\left(  t\right)
=\varepsilon\omega\left(  \varepsilon t\right)  ,$ $\varepsilon>0.$
\textit{Generalized weak abelian wave operator} $\widetilde{\mathfrak{U}}%
_{\pm}$ is defined by the equality
\[
\widetilde{\mathfrak{U}}_{\pm}\left(  H_{1},H;J\right)  =w.o.t\text{-}%
\lim_{\varepsilon\rightarrow0}\int_{0}^{\infty}\omega_{\varepsilon}\left(
t\right)  P_{ac}^{\infty}\left(  H_{1}\right)  W_{J}\left(  \pm t\right)
P_{ac}^{\infty}\left(  H\right)  dt
\]
where $w.o.t.$ stands for the weak operator topology. So similarly to the
argument in \cite{Y1}, we conclude that $\widetilde{\mathfrak{U}}_{\pm}\left(
H_{1},H;J\right)  =\widetilde{W}_{\pm}\left(  H_{1},H;J\right)  $ if
$\widetilde{W}_{\pm}\left(  H_{1},H;J\right)  $ exists. If we set
$\omega_{\varepsilon}\left(  t\right)  =2\varepsilon\exp\left(  -2\varepsilon
t\right)  ,$ then for a pair of elements $f$ and $f_{1}$ in $\mathcal{H},$
\[
\left\langle \widetilde{\mathfrak{U}}_{\pm}\left(  H_{1},H;J\right)
f,f_{1}\right\rangle =\lim_{\varepsilon\rightarrow0}2\varepsilon\int
_{0}^{\infty}e^{-2\varepsilon t}\left\langle W_{J}\left(  t\right)
P_{ac}^{\infty}\left(  H\right)  f,P_{ac}^{\infty}\left(  H_{1}\right)
f_{1}\right\rangle dt
\]
if $\widetilde{\mathfrak{U}}_{\pm}\left(  H_{1},H;J\right)  $ exists.

\begin{remark}
\label{A2}Let $\theta\left(  t\right)  =1$ as $t\geq0$ and $\theta\left(
t\right)  =0$ for $t<0.$ We apply Placherel's Theorem for Fourier
transformation on Hilbert space (see Theorem 1.8.1 in \cite{Arendt}) to the
functions
\[
\theta e^{-\varepsilon t}JU_{H}\left(  \pm t\right)  P_{ac}^{\infty}\left(
H\right)  f,\text{ \ }\theta e^{-\varepsilon t}JU_{H_{1}}\left(  \pm t\right)
P_{ac}^{\infty}\left(  H_{1}\right)  f_{1}.
\]
Considering the expression in (\ref{g3}) for the resolvents of $H$ and
$H_{1},$ we get that
\begin{align*}
&  2\varepsilon\int_{0}^{\infty}e^{-2\varepsilon t}\left\langle JU_{H}\left(
\pm t\right)  P_{ac}^{\infty}\left(  H\right)  f,U_{H_{1}}\left(  \pm
t\right)  P_{ac}^{\infty}\left(  H_{1}\right)  f_{1}\right\rangle dt\\
&  =\frac{\varepsilon}{\pi}\int_{-\infty}^{\infty}\left\langle JR_{H}\left(
\lambda\pm i\varepsilon\right)  P_{ac}^{\infty}\left(  H\right)  f,R_{H_{1}%
}\left(  \lambda\pm i\varepsilon\right)  P_{ac}^{\infty}\left(  H_{1}\right)
f_{1}\right\rangle d\lambda.
\end{align*}

\end{remark}

From the above argument, we can get the following result.

\begin{lemma}
\label{A3}Let $H$, $H_{1}\in\mathcal{A}\left(  \mathcal{M}\right)  $ be a pair
of self-adjoint operators and $J$ be an operator in $\mathcal{M}$. The
existence of $\widetilde{\mathfrak{U}}_{\pm}\left(  H_{1},H;J\right)  $ is
equivalent to the existence of the following limit
\[
\lim_{\varepsilon\rightarrow0}\frac{\varepsilon}{\pi}\int_{-\infty}^{\infty
}\left\langle JR_{H}\left(  \lambda\pm i\varepsilon\right)  P_{ac}^{\infty
}\left(  H\right)  f,R_{H_{1}}\left(  \lambda\pm i\varepsilon\right)
P_{ac}^{\infty}\left(  H_{1}\right)  f_{1}\right\rangle d\lambda
\]
for any $f$ and $f_{1}$ in $\mathcal{H}$. Furthermore,
\begin{align*}
&  \left\langle \widetilde{\mathfrak{U}}_{\pm}\left(  H_{1},H;J\right)
f,f_{1}\right\rangle \\
&  =\lim_{\varepsilon\rightarrow0}\frac{\varepsilon}{\pi}\int_{-\infty
}^{\infty}\left\langle JR_{H}\left(  \lambda\pm i\varepsilon\right)
P_{ac}^{\infty}\left(  H\right)  f,R_{H_{1}}\left(  \lambda\pm i\varepsilon
\right)  P_{ac}^{\infty}\left(  H_{1}\right)  f_{1}\right\rangle d\lambda.
\end{align*}

\end{lemma}

The generalized stationary wave operator of $H$ and $H_{1}$ in $\mathcal{M}$
is given below.

\begin{lemma}
\label{A6}(Theorem 1.1.3, \cite{Y1}) Suppose for functions $f_{\varepsilon}\in
L_{1}\left(  \mathbb{R}\right)  $ the integrals
\[
\int_{X}f_{\varepsilon}\left(  \lambda\right)  d\lambda
\]
tend to zero uniformly with respect to $\varepsilon$ as $\left\vert
X\right\vert \rightarrow0.$ Suppose also the same for $X=(-\infty
,-N)\cup(N,\infty)$ as $N\rightarrow\infty,$ and
\[
\lim_{\varepsilon\rightarrow0}f_{\varepsilon}\left(  \lambda\right)  =f\left(
\lambda\right)  \text{ a.e. }\lambda\in\mathbb{R},
\]
then $f\in L_{1}\left(  \mathbb{R}\right)  $ and
\[
\lim_{\varepsilon\rightarrow0}\int_{-\infty}^{\infty}f_{\varepsilon}\left(
\lambda\right)  d\lambda=\int_{-\infty}^{\infty}f(\lambda)d\lambda.
\]

\end{lemma}

\begin{theorem}
\label{A7}Let $H$, $H_{1}\in\mathcal{A}\left(  \mathcal{M}\right)  $ be a pair
of self-adjoint operators and $J$ be an operator in $\mathcal{M}$. If
\[
\lim_{\varepsilon\rightarrow0}\frac{\varepsilon}{\pi}\left\langle
JR_{H}\left(  \lambda\pm i\varepsilon\right)  P_{ac}^{\infty}\left(  H\right)
f,R_{H_{1}}\left(  \lambda\pm i\varepsilon\right)  P_{ac}^{\infty}\left(
H_{1}\right)  f_{1}\right\rangle
\]
exists for any pair of $f$ and $f_{1}$ in $\mathcal{H}$ a.e. $\lambda
\in\mathbb{R}$, then $\mathcal{U}_{\pm}\left(  H_{1},H;J\right)  $ is
well-defined and bounded with $\left\Vert \mathcal{U}_{\pm}\left(
H_{1},H;J\right)  \right\Vert \leq\left\Vert J\right\Vert .$ Meanwhile
$\widetilde{\mathfrak{U}}_{\pm}\left(  H_{1},H;J\right)  $ also exists under
this condition and
\[
\mathcal{U}_{\pm}\left(  H_{1},H;J\right)  =\widetilde{\mathfrak{U}}_{\pm
}\left(  H_{1},H;J\right)  .
\]
Furthermore,
\[
\mathcal{U}_{\pm}\left(  H_{1},H;J\right)  =\widetilde{W}_{\pm}(H_{1},H;J)
\]
if $\widetilde{W}_{\pm}(H_{1},H;J)$ exists.
\end{theorem}

\begin{proof}
Since $P_{ac}^{\infty}\left(  H\right)  \leq P_{ac}\left(  H\right)  $,
$P_{ac}^{\infty}\left(  H_{1}\right)  \leq P_{ac}\left(  H_{1}\right)  $ and
\[
\lim\frac{\varepsilon}{\pi}\left\langle JR_{H}\left(  \lambda\pm
i\varepsilon\right)  P_{ac}^{\infty}\left(  H\right)  f,R_{H_{1}}\left(
\lambda\pm i\varepsilon\right)  P_{ac}^{\infty}\left(  H_{1}\right)
f_{1}\right\rangle
\]
exists for any pair of $f$ and $f_{1}$ in $\mathcal{H}$ a.e. $\lambda
\in\mathbb{R}$, then by the definition of $\mathcal{U}_{\pm}\left(
H_{1},H;J\right)  ,$ we get that
\begin{align*}
&  \left\vert \left\langle \mathcal{U}_{\pm}\left(  H_{1},H;J\right)
f,f_{1}\right\rangle \right\vert \\
&  =\left\vert \frac{1}{\pi}\int_{-\infty}^{\infty}\lim_{\varepsilon
\rightarrow0}\varepsilon\left\langle JR_{H}\left(  \lambda\pm i\varepsilon
\right)  P_{ac}^{\infty}\left(  H\right)  f,R_{H_{1}}\left(  \lambda\pm
i\varepsilon\right)  P_{ac}^{\infty}\left(  H_{1}\right)  f_{1}\right\rangle
d\lambda\right\vert \\
&  \leq\left\Vert J\right\Vert \left(  \int_{-\infty}^{\infty}\lim
_{\varepsilon\rightarrow0}\left(  \frac{\varepsilon}{\pi}\left\Vert
R_{H}\left(  \lambda\pm i\varepsilon\right)  P_{ac}^{\infty}\left(  H\right)
f\right\Vert ^{2}\right)  d\lambda\right)  ^{1/2}\left(  \int_{-\infty
}^{\infty}\lim_{\varepsilon\rightarrow0}\left(  \frac{\varepsilon}{\pi
}\left\Vert R_{H_{1}}\left(  \lambda\pm i\varepsilon\right)  P_{ac}^{\infty
}\left(  H_{1}\right)  f_{1}\right\Vert ^{2}\right)  d\lambda\right)  ^{1/2}\\
&  \leq\left\Vert J\right\Vert \left(  \int_{-\infty}^{\infty}\frac{d\left(
E_{H}\left(  \lambda\right)  P_{ac}^{\infty}\left(  H\right)  f,f\right)
}{d\lambda}d\lambda\cdot\int_{-\infty}^{\infty}\frac{d\left(  E_{H_{1}}\left(
\lambda\right)  P_{ac}^{\infty}\left(  H_{1}\right)  f_{1},f_{1}\right)
}{d\lambda}d\lambda\right)  ^{1/2}\\
&  \leq\left\Vert J\right\Vert \left\Vert P_{ac}^{\infty}\left(  H\right)
f\right\Vert \left\Vert P_{ac}^{\infty}\left(  H_{1}\right)  f_{1}\right\Vert
\leq\left\Vert J\right\Vert \left\Vert f\right\Vert \left\Vert f_{1}%
\right\Vert
\end{align*}
by Definition \ref{R5}, equalities (\ref{g21}) and (\ref{g6}). Hence
$\mathcal{U}_{\pm}\left(  H_{1},H;J\right)  $ exists and is bounded with
\[
\left\Vert \mathcal{U}_{\pm}\left(  H_{1},H;J\right)  \right\Vert
\leq\left\Vert J\right\Vert .
\]

For a fixed positive number $\varepsilon_{0}$ and any Borel set $Y\subseteq
\mathbb{R},$ from Fubini's Theorem and equality (\ref{g40}), we have%
\begin{align*}
\left\vert \int_{Y}\left\langle \delta_{H}\left(  \lambda,\varepsilon\right)
f,g\right\rangle d\lambda\right\vert  &  =\left\vert \int_{Y}\frac
{\varepsilon}{\pi}\left(  \int_{-\infty}^{\infty}\frac{1}{\left(
s-\lambda-i\varepsilon\right)  \left(  s-\lambda+i\varepsilon\right)
}d\left\langle E_{H}\left(  s\right)  f,g\right\rangle \right)  d\lambda
\right\vert \\
&  =\left\vert \int_{-\infty}^{\infty}\frac{\varepsilon}{\pi}\left(  \int
_{Y}\frac{1}{(s-\lambda)^{2}+\varepsilon^{2}}d\lambda\right)  d\left\langle
E_{H}\left(  s\right)  f,g\right\rangle \right\vert \\
&  \leq\int_{-\infty}^{\infty}\frac{\varepsilon}{\pi}\left(  \int_{Y}\frac
{1}{(s-\lambda)^{2}+\varepsilon_{0}^{2}}d\lambda\right)  d\left\langle
E_{H}\left(  s\right)  f,g\right\rangle \rightarrow0
\end{align*}
uniformly with respect to $\varepsilon\geq\varepsilon_{0}$ as $\left\vert
Y\right\vert \rightarrow0$ or $Y=\left(  -\infty,-N\right)  \cup\left(
N,+\infty\right)  $ with $N\rightarrow\infty.$ Therefore by equality
(\ref{g21})
\begin{align*}
&  \left\vert \frac{1}{\pi}\int_{Y}\varepsilon\left\langle JR_{H}\left(
\lambda\pm i\varepsilon\right)  P_{ac}^{\infty}\left(  H\right)  f,R_{H_{1}%
}\left(  \lambda\pm i\varepsilon\right)  P_{ac}^{\infty}\left(  H_{1}\right)
f_{1}\right\rangle d\lambda\right\vert \\
&  \leq\left\Vert J\right\Vert \left(  \int_{-\infty}^{\infty}\frac
{\varepsilon}{\pi}\left\Vert R_{H_{1}}\left(  \lambda\pm i\varepsilon\right)
P_{ac}^{\infty}\left(  H_{1}\right)  f_{1}\right\Vert ^{2}d\lambda\int
_{Y}\frac{\varepsilon}{\pi}\left\Vert R_{H}\left(  \lambda\pm i\varepsilon
\right)  P_{ac}^{\infty}\left(  H\right)  f\right\Vert ^{2}d\lambda\right)
^{1/2}\\
&  \leq\left\Vert J\right\Vert \left\Vert f\right\Vert \left(  \int
_{Y}\left\langle \delta_{H}\left(  \lambda,\varepsilon\right)  P_{ac}^{\infty
}\left(  H\right)  f,f\right\rangle d\lambda\right)  ^{1/2}\rightarrow0
\end{align*}
uniformly with respect to $\varepsilon\geq\varepsilon_{0}$ as $\left\vert
Y\right\vert \rightarrow0$ or $Y=\left(  -\infty,-N\right)  \cup\left(
N,+\infty\right)  $ with $N\rightarrow\infty.$

For $P_{ac}^{\infty}\left(  H\right)  f\in\mathcal{H}_{ac}\left(  H\right)  ,$
we have
\[
\lim_{\varepsilon\rightarrow0}\left\langle \delta_{H}\left(  \lambda
,\varepsilon\right)  P_{ac}^{\infty}\left(  H\right)  f,f\right\rangle
=\frac{d\left(  E_{H}\left(  \lambda\right)  P_{ac}^{\infty}\left(  H\right)
f,f\right)  }{d\lambda}\in L_{1}\left(  \mathbb{R}\right)
\]
by equality (\ref{g6}). Then now from Lemma \ref{A6} and the assumption of the
existence of $\mathcal{U}_{\pm}\left(  H_{1},H;J\right)  $, we have
\begin{align*}
&  \frac{1}{\pi}\lim_{\varepsilon\rightarrow0}\int_{-\infty}^{\infty
}\varepsilon\left\langle JR_{H}\left(  \lambda\pm i\varepsilon\right)
P_{ac}^{\infty}\left(  H\right)  f,R_{H_{1}}\left(  \lambda\pm i\varepsilon
\right)  P_{ac}^{\infty}\left(  H_{1}\right)  f_{1}\right\rangle d\lambda\\
&  =\frac{1}{\pi}\int_{-\infty}^{\infty}\lim_{\varepsilon\rightarrow
0}\varepsilon\left\langle JR_{H}\left(  \lambda\pm i\varepsilon\right)
P_{ac}^{\infty}\left(  H\right)  f,R_{H_{1}}\left(  \lambda\pm i\varepsilon
\right)  P_{ac}^{\infty}\left(  H_{1}\right)  f_{1}\right\rangle d\lambda.
\end{align*}
It implies that $\widetilde{\mathfrak{U}}_{\pm}\left(  H,H_{1};J\right)  $
exists and
\[
\left\langle \mathcal{U}_{\pm}\left(  H_{1},H;J\right)  f,f_{1}\right\rangle
=\left\langle \widetilde{\mathfrak{U}}_{\pm}\left(  H_{1},H;J\right)
f,f_{1}\right\rangle
\]
for any pair of $f$ and $f_{1}$ in $\mathcal{H}.$ Since the equality
$\widetilde{\mathfrak{U}}_{\pm}\left(  H_{1},H;J\right)  =\widetilde{W}_{\pm
}(H_{1},H;J)$ holds if $\widetilde{W}_{\pm}(H_{1},H;J)$ exists, we can quickly
get
\[
\mathcal{U}_{\pm}\left(  H_{1},H;J\right)  =\widetilde{W}_{\pm}(H_{1},H;J)
\]
if $\widetilde{W}_{\pm}(H_{1},H;J)$ exists.
\end{proof}

Next result give us the relation among $\mathcal{U}_{\pm}\left(
H_{1},H;J\right)  ,$ $\widetilde{W}_{\pm}\left(  H_{1},H;J\right)  $ and
$W_{\pm}(H_{1},H;J)$.

\textbf{Theorem 3.2.4} Let $H$, $H_{1}\in\mathcal{A}\left(  \mathcal{M}%
\right)  $ be a pair of self-adjoint operators and $J$ be an operator in
$\mathcal{M}$. If $\ \mathcal{U}_{\pm}\left(  H_{1},H;J\right)  $,
$\mathcal{U}_{\pm}\left(  H,H;J^{\ast}J\right)  $, $\widetilde{W}_{\pm}\left(
H_{1},H;J\right)  $ and $\widetilde{W}_{\pm}\left(  H,H;J^{\ast}J\right)  $
exist and
\[
\mathcal{U}_{\pm}^{\ast}\left(  H_{1},H;J\right)  \mathcal{U}_{\pm}\left(
H_{1},H;J\right)  =\mathcal{U}_{\pm}\left(  H,H;J^{\ast}J\right)  ,
\]
then $W_{\pm}(H_{1},H;J)$ exists and
\[
\mathcal{U}_{\pm}\left(  H_{1},H;J\right)  =W_{\pm}(H_{1},H;J).
\]

\begin{proof}
By Theorem \ref{A7}, if
\[
\mathcal{U}_{\pm}\left(  H,H_{1};J\right)  ,\text{ }\mathcal{U}_{\pm}\left(
H,H;J^{\ast}J\right)  ,\text{ }\widetilde{W}_{\pm}\left(  H_{1},H;J\right)
\text{ and }\widetilde{W}_{\pm}\left(  H,H;J^{\ast}J\right)
\]
exist, then
\[
\mathcal{U}_{\pm}\left(  H_{1},H;J\right)  =\widetilde{W}_{\pm}\left(
H_{1},H;J\right)  ,
\]%
\[
\mathcal{U}_{\pm}^{\ast}\left(  H_{1},H;J\right)  =\widetilde{W}_{\pm}^{\ast
}\left(  H_{1},H;J\right)
\]
and
\[
\mathcal{U}_{\pm}\left(  H,H;J^{\ast}J\right)  =\widetilde{W}_{\pm}\left(
H,H;J^{\ast}J\right)  .
\]
So by Lemmar \ref{A1} and equality
\[
\mathcal{U}_{\pm}^{\ast}\left(  H_{1},H;J\right)  \mathcal{U}_{\pm}\left(
H_{1},H;J\right)  =\mathcal{U}_{\pm}\left(  H,H;J^{\ast}J\right)  ,
\]
we know $W_{\pm}(H_{1},H;J)$ exists$.$ The proof is completed.
\end{proof}

\vspace{3cm}

\end{document}